\documentclass[11pt]{amsart}

\usepackage[T1]{fontenc}
\usepackage[utf8]{inputenc}
\usepackage{lmodern}

\usepackage{xcolor}

\definecolor{Parhamcomment}{RGB}{0, 0, 255}
\definecolor{mygreen}{RGB}{0, 128, 0}
\definecolor{myred}{RGB}{255, 0, 0}
\definecolor{Parhamtext}{rgb}{0.0, 0.58, 0.71}

\definecolor{refers}{rgb}{1.0, 0.6, 0.4}
\definecolor{calcs}{rgb}{0.8, 0.6, 0.9}
\definecolor{ques}{rgb}{0.03, 0.91, 0.87}
\definecolor{late}{rgb}{1.0, 0.49, 0.0}

\usepackage{amsmath, amssymb, amsthm, amsaddr}
\usepackage{amsfonts} 
\usepackage{mathrsfs} 

\usepackage{graphicx}
\usepackage{caption}
\usepackage{subcaption}
\usepackage{tikz}
\usepackage{float}
\usetikzlibrary{decorations.markings}
\usetikzlibrary{matrix}
\usepackage{mathabx,epsfig}
\def\acts{\mathrel{\reflectbox{$\righttoleftarrow$}}}

\usepackage{enumitem} 
\usepackage{hyperref} 
\usepackage{bm} 
\usepackage{siunitx} 
\usepackage{booktabs} 
\usepackage{url} 
\usepackage{microtype} 
\usepackage{fp} 
\usepackage{xfp} 
\usepackage{physics}
\usepackage{comment}
\usepackage{verbatim}

\usepackage[margin=1in]{geometry}

\newtheorem{theorem}{Theorem}[section]
\newtheorem{definition}[theorem]{Definition}
\newtheorem{proposition}[theorem]{Proposition}

\newtheorem{lemma}[theorem]{Lemma}
\newtheorem{remark}[theorem]{Remark}
\newtheorem{example}[theorem]{Example}

\newcommand{\R}{\mathbb{R}}

\newcommand{\C}{\mathbb{C}}

\newcommand{\E}{\mathbb{E}}

\renewcommand{\epsilon}{\varepsilon}
\renewcommand{\phi}{\varphi}

\newcommand{\lnorm}{\lvert \!\vert}
\newcommand{\rnorm}{\vert \! \rvert}

\DeclareMathOperator{\Id}{Id}

\renewcommand{\P}{\mathbb{P}}

\DeclareRobustCommand{\ammaG}{\text{\reflectbox{$\Gamma$}}}

\newcommand{\diff}{\mathop{}\!\mathrm{d}}
\DeclareMathOperator{\Cov}{Cov}

\DeclareMathOperator{\Crt}{Crt}

\newcommand{\injnorm}[1]{\lnorm #1 \rnorm_{\mathrm{inj}}}

\title{On the injective norm of random fermionic states and skew-symmetric tensors}
\author{Stephane Dartois}
\address{\'Ecole Polytechnique, Institut Polytechnique de Paris, Centre de Mathématiques Laurent Schwartz, Palaiseau, France}
\email{stephane.dartois@polytechnique.edu}
\author{Parham Radpay}
\address{Université Paris-Saclay, CEA, Palaiseau, France}
\email{parham.radpay@cea.fr}

\begin{document}
\maketitle
\begin{abstract}
We study the injective norm of random skew-symmetric tensors and the associated fermionic quantum states, a natural measure of multipartite entanglement for systems of indistinguishable particles. Extending recent advances on random quantum states, we analyze both real and complex skew-symmetric Gaussian ensembles in two asymptotic regimes: fixed particle number with increasing one-particle Hilbert space dimension, and joint scaling with fixed filling fraction. Using the Kac-Rice formula on the Grassmann manifold, we derive high-probability upper bounds on the injective norm and establish sharp asymptotics in both regimes. Interestingly, a duality relation under particle–hole transformation is uncovered, revealing a symmetry of the injective norm under the action of the Hodge star operator. We complement our analytical results with numerical simulations for low fermion numbers, which matches the predicted bounds.
\end{abstract}
\section{Introduction and main results}

\noindent{\bf Setup:} In this paper we study the injective norm of random skew-symmetric tensor $T\in \bigwedge^p \mathbb{K}^d\subseteq (\mathbb{K}^d)^{\otimes p}$, where $\mathbb{K}=\R$ or $\C$.  A skew-symmetric tensor is determined by a set of $d \choose p$ coordinates, as in fact, for $\{e_i\}_{i=1}^d$ a basis of $\mathbb{K}^d$,
$$T=\sum_{i_1,\ldots, i_p=1}^d T_{i_1,\ldots, i_p}e_{i_1}\otimes \ldots \otimes e_{i_p},$$
 where $T_{i_{\sigma(1)},\ldots,i_{\sigma(p)}}=\text{sg}(\sigma) T_{i_1,\ldots, i_p}$, for all $\sigma \in \mathfrak{S}_p$ the permutation group over $p$ elements, and $\text{sg}:\mathfrak{S}_p\to \{\pm 1\}$ the signature function. In particular, entries with two or more identical indices vanish. If $\mathbb{K}=\C$ to each $T$ we associate a quantum state of fermions $\ket{\psi_f}:=\frac{T}{\lnorm T\rnorm_{2}},$ where $\lnorm \cdot \rnorm_{2}$ denotes the Hilbert-Schmidt (or simply $2-$) norm, that is 
 $$\lnorm T\rnorm_{2}^2 = \sum_{i_1,\ldots, i_p=1}^d\lvert T_{i_1,\ldots, i_p}\rvert^2.$$
Furthermore, one has 
$$T=\sum_{i_1<\ldots<i_p}T_{i_1,\ldots, i_p} e_{i_1}\wedge\ldots\wedge e_{i_p},$$ 
where $\wedge$ denotes the wedge-product of vectors.  We are interested in the injective norm for both $T$ and $\ket{\psi_f}$ as defined by
\begin{equation}
    \injnorm{T}:=\max_{\substack{x_1, x_2, \ldots, x_p\in \mathbb{K}^d\\ \lnorm x_1\wedge x_2\wedge \ldots \wedge x_p\rnorm_2=1}}\lvert\langle T,x_1\wedge x_2\wedge \ldots \wedge x_p\rangle \rvert,
\end{equation}
$\lnorm x_1\wedge x_2\wedge \ldots \wedge x_p\rnorm_2$ is induced by the Hilbert-Schmidt norm  on $\mathbb{K}^d$ and the same definition applies to the quantum state $\ket{\psi_f}$
\begin{equation}
    \injnorm{\ket{\psi_f}}:=\max_{\substack{x_1, x_2, \ldots, x_p\in \mathbb{K}^d\\ \lnorm x_1\wedge x_2\wedge \ldots \wedge x_p\rnorm_2=1}}\lvert\braket{\psi_f}{ x_1\wedge x_2\wedge \ldots \wedge x_p}\rvert.
\end{equation}
In this context, the geometric measure of entanglement remains 
$$\text{GME}(\ket{\psi_f})=-2\log \injnorm{\ket{\psi_f}}.$$
\begin{remark}
This definition is the most natural since strictly separable skew-symmetric states do not exist. Moreover, it follows from the usual definition. In fact, for all $\sigma\in \mathfrak{S}_p$, the permutation group over $p$ elements, let $\sigma\acts (\mathbb{K}^d)^{\otimes p}$ by exchanging tensor factors, that is $\sigma:x_1\otimes\ldots\otimes x_p\mapsto x_{\sigma(1)}\otimes \ldots\otimes x_{\sigma(p)}$. Denote 
$$P_A:(\mathbb{K}^d)^{\otimes p}\to \bigwedge^p \mathbb{K}^d,\ P_A:=\frac1{k!}\sum_{\sigma\in \mathfrak{S}_p}\text{sg}(\sigma)\sigma$$ the projector onto the skew-symmetric subspace. Then, for all $T\in \bigwedge^p \mathbb{K}^d,$ $P_A(T)=T$. Therefore, coming back to the usual definition\cite{aubrun2017alice, dartois2024injective} of the injective norm one has, using the skew-symmetry of $T$,
\begin{align*}
\max_{x_1,\ldots, x_p\in \mathbb{K}^d,\lnorm x_i\rnorm_2=1} \lvert\langle T,x_1\otimes x_2\otimes \ldots \otimes x_p\rangle \rvert &=\lvert\langle T,x_1\otimes x_2\otimes \ldots \otimes x_p\rangle \rvert\\
&=\lvert\langle P_A(T),x_1\otimes x_2\otimes \ldots \otimes x_p\rangle \rvert\\
&=\lvert\langle T,P_A(x_1\otimes x_2\otimes \ldots \otimes x_p)\rangle \rvert=\lvert\langle T,x_1\wedge x_2\wedge \ldots \wedge x_p\rangle \rvert,
\end{align*}
by definition of the wedge product. 
Additionally, the set of simple wedge products has already been formalized as the natural set of separable states for fermions, see for instance \cite{grabowski2012segre}, where more generally, para-statistics are considered. 
\end{remark}

\noindent{\bf Motivations and main result: } Different notions of entanglement for multipartite states have been studied in the quantum information community. Among them, the geometric entanglement has been extensively investigated in the literature~\cite{shimony1995degree,wei2003geometric,aulbach2010maximally,aulbach2010geometric,orus2014geometric,aubrun2017alice,fitter2022estimating, steinberg2024finding,dartois2025injectivenormcssquantum}.  A central aim of this work is to extend the analysis performed in~\cite{dartois2024injective}, which determines (lower bounds on) the geometric measure of entanglement of random quantum states of distinguishable-particles, to random states of fermions. We treat two different scaling regimes:
\begin{enumerate}
    \item fixed-$p$ fermion number while letting the one-particle Hilbert space dimension $d\to \infty$.
    \item a double-scaling regime in which $p,d\to \infty$ while the filling fraction $\alpha=p/d$ of the random states is fixed asymptotically.
\end{enumerate}   
Following the approach of the first author in the article~\cite{dartois2024injective}, we study both real and complex skew-symmetric tensors. This puts our results within, and extends to multipartite fermionic entanglement, a line of work on bipartite entanglement measures for random fermionic states~\cite{BianchiKieburg2021,dartois2022entanglement,pastur2024entanglement,huang2023entropy}. \\

The main results of the paper are as follows:
\begin{theorem}
    Let $T$ be a skew-symmetric random tensor over the field $\mathbb{K}$, with i.i.d entries up to antisymmetry, with entries distributed as $\mathcal{N}_\R(0,1)$ in the case $\mathbb{K}=\R$ and $\mathcal{N}_\C(0,1)$ in the case $\mathbb{K}=\C$. And let $\ket{\psi_f}:=\frac{T}{\lnorm T\rnorm_{2}}$, where $\lnorm T\rnorm_{2}$ is the Hilbert-Schmidt norm of the tensor $T$. Then we  have for fixed $p$:
    \begin{align}
        \limsup_{d\rightarrow \infty}\frac1{p(d-p)}\log \P\left(\frac1{\sqrt{p(d-p)}}\injnorm{T}> E_0(p)+\epsilon\right)<0,
    \end{align}
    \begin{align}
        \limsup_{d\rightarrow \infty}\frac1{p(d-p)}\log \P\left(\frac{\injnorm{\ket{\psi_f}}}{\sqrt{p(d-p)}}>\frac1{d^{\frac{p}{2}}}(E_0(p)+\epsilon)\right)<0.
    \end{align}
    and for a double scaling regime where both $d,p \rightarrow \infty$, with $p =\lfloor \alpha  d \rfloor$ and a fixed $\alpha \in (0,1)$:
    \begin{align}
    \limsup_ {d \rightarrow \infty} \frac{1}{p(d-p)}\log \P\left( \frac{\injnorm{T}}{\sqrt{d^2\alpha(1-\alpha)}} >\gamma_\alpha(d)+ \frac{\epsilon}{\sqrt{\log d}} \right) < 0,
\end{align}
\begin{align}
    \limsup_{d\rightarrow \infty}\frac1{p(d-p)}\log \P\left(\frac{\injnorm{ \ket{\psi_f}}}{\sqrt{d^2\alpha(1-\alpha)}}> d^{-\frac{\alpha d}{2}} e^{\frac12(\alpha + (1-\alpha)\log(1-\alpha))d} (\gamma_\alpha(d)+\epsilon)\right)<0,
\end{align}
with $\gamma_\alpha(d)$ and $E_0(p)$ defined in \ref{lem: Asympgamma} and \ref{lem:annealed_complexity_real} respectively. 
\end{theorem}
Interestingly, we also uncover a duality relation for the geometric measure of entanglement of fermionic states. The GME, like the injective norm, is invariant under particle-hole or Hodge duality. This duality is recovered in the Kac-Rice integral and translates in the particular block matrix structure of the random Hessian. In computing the limiting injective norm, we recover this duality in terms of $\alpha$ as a symmetry of the entanglement measure under the transformation $\alpha \leftrightarrow 1-\alpha$. \\

From a random-matrix-to-random-tensors point of view, the injective norm of a random tensor generalizes the operator norm of a random matrix. It is therefore natural to study the behavior of the injective norm for different random tensor ensembles, as it is natural to study the behavior of the operator norm for different ensembles of random matrices. In the random matrix world, several symmetry classes are investigated \cite{mehta2004random,forrester2010log}: real asymmetric, real symmetric, real skew-symmetric \cite{kuriki2010distributions}, complex asymmetric, Hermitian. Those can display properties driven by their symmetry. In the tensor world however, results concentrate on the symmetric real or real tensors~\cite{nguyen2010tensor,auffinger2013random,tomioka2014spectral, bandeira2024geometric, boedihardjo2024injective}. But, symmetry-wise, the situation is richer. Symmetry classes of tensors of order $p$ are in bijection to irreducible representations of the symmetric group $\mathfrak{S}_p$. Hence, it is natural to investigate the properties of random tensors in various symmetry classes.
Under this light, the case of skew-symmetric Gaussian tensors is analogous, at least in the real case, to skew-symmetric Gaussian random matrices\footnote{which are themselves seldom studied, despite being a very natural Gaussian $O(N)$ invariant ensemble}. Our result can be seen as part of a program aimed at understanding the asymptotic properties of the injective norm of random tensors for all their possible symmetry classes.  To fully achieve such a program, one should certainly develop systematic methods to obtain asymptotically sharp bounds. While this remains an open direction, the present work provides another example of the computation of a (conjecturally) sharp upper bound for a non-symmetric (here, skew-symmetric) Gaussian tensor ensemble.\\

The study of the injective norm is closely related to the study of the ground state in certain disordered systems, notably spin glasses. The first steps in studying the complexity of these systems and consequently their ground states using the Kac-Rice formula were taken in \cite{auffinger2013random}. In \cite{auffinger2013random} the authors provided precise asymptotic account of critical points at given energy levels and revealed a complex structure near the ground state.\\
Building on these approaches and combining them with more recent random matrix results developed in \cite{arous2022exponential, mckenna2024complexity, bandeira2023matrix}, the work \cite{dartois2024injective} studied what, in the disordered systems world, could be considered the ground state of multispecies (eventually complex-valued) spin-glasses and interpreted it as an injective norm or geometric measure of entanglement, thus connecting quantum information with the former domains. High probability upper bounds were obtained in two asymptotic regimes in the case of non-symmetric real and complex tensors with independent Gaussian i.i.d. entries: the first regime being fixed-$p$ and $d\rightarrow \infty$ and the second being fixed-$d$ and $p\rightarrow \infty$. One of those upper bounds was later matched with a lower bound in the real case, fixed-$p$ and $d\rightarrow \infty$ regime in \cite{bates2025balanced}.\\

From the point of view of Kac-Rice formula, and the landscape complexity program, one novel aspect of our work is that we use the formula on the Grassmann manifold $\text{Gr}(p,d)$. This is one of several reasons why the relevant Gaussian process escapes the conditional results of \cite{Sub2023} coming from the spin glass literature in the real case (and even more so in the complex case) as was already the case of the earlier work in \cite[Remark 3.9]{dartois2024injective}. It also escapes more recent results such as \cite{bates2025balanced, stojnic2025ground}.\\

In parallel \cite{bandeira2024geometric} used asymptotic geometric analysis to produce upper bounds on $\ell_p$-injective norms and general structured Gaussian tensors. Additionally, \cite{boedihardjo2024injective} provided bounds for the injective norm of random Gaussian tensors with independent entries having a variance profile. 
Both works extend well-known random matrix results to random tensors.  \\

Very recently, the physics work \cite{delporte2025real} studied what they call the (signed) eigenvalue distribution of random real skew-symmetric tensors using quantum field theory related methods in the large one-particle Hilbert spaces dimension regime. This provides  heuristics for the typical value of the injective norm in this regime. This very recent work can be seen as extending the previous physics work \cite{sasakura2024signed}.\\

\noindent{\bf Organization:} The paper is organized in a way that mostly mimics the organization of \cite{dartois2024injective}. In fact, the proof of our main theorem is built around very similar ideas. In Section 2, we introduce the definitions related to quantum information and random matrix theory. We also present some fundamental properties of the defined concepts, which we use frequently in the paper. In Section 3, we prove the results for the real case in both asymptotic regimes. Section 4 deals with the proofs in the complex case. The proofs follow the line of logic of the real case, with only small differences. All proofs in these two sections rely heavily on \cite{alt2019location,arous2022exponential,bandeira2023matrix,dartois2024injective}. Finally, in Section 5 we present some numerical simulations that match our analytical results in the first asymptotic regime. This section can be seen as an extension of the work \cite{fitter2022estimating}. The appendix contains details about the geometric properties of the Grassmann manifold and the calculation of the correlation functions necessary in the Kac-Rice formula.
\section*{Acknowledgements}
S.D. is grateful to the Institut de Mathématiques de Bordeaux for their hospitality, where part of this work was carried out. The work of S.D. was partly supported by the ANR grants ANR-25-CE40-1380 and ANR-25-CE40-5672.
\section{Preliminaries}
\subsection{Random matrix models and their properties}
We here recall the definition of the BHGOE ensemble introduced in \cite{dartois2024injective}. Building on that we introduce the BHGAE ensemble and show that their operator norm is bounded from above.
\begin{definition}[Injective norm of a skew-symmetric tensor]Let $\mathbb{K}$ be $\C$ or $\R$ for the remainder of the paper.
The injective norm of a skew-symmetric tensor  $T \in \left( \mathbb{K}^d \right)^{\otimes p}$ can be defined by:
\begin{align} \label{eq:Injective_norm}
\injnorm{{T}}&=\max_{\substack{x^{(1)},\ldots, x^{(p)}\\ \lnorm x^{(1)}\wedge \ldots \wedge x^{(p)}\rnorm_2=1}}|\langle T, x^{(1)}\wedge \ldots  \wedge x^{(p)}\rangle|=\max_{\substack{x^{(1)},\ldots, x^{(p)}\\ \lnorm x^{(1)}\wedge \ldots \wedge x^{(p)}\rnorm_2\le1}}|\langle T, x^{(1)}\wedge \ldots\wedge x^{(p)}\rangle|\\ \nonumber &=\max_{\substack{x^{(1)},\ldots, x^{(p)}}}\frac1{\lnorm x^{(1)}\wedge \ldots \wedge x^{(p)}\rnorm_2}|\langle T, x^{(1)}\wedge \ldots \wedge x^{(p)}\rangle|,
\end{align}  
where $\lnorm x^{(1)} \wedge \dots \wedge x^{(p)} \rnorm$ is shown to be: 
\begin{align}
    \lnorm x^{(1)} \wedge \dots \wedge x^{(p)} \rnorm_2^2 = |\det(X^tX)|
\end{align}
with $X$ being a $d \times p$ matrix with its columns given by $x^{(1)},x^{(2)}, \ldots,x^{(p)}$.
\end{definition}
\noindent To turn $T$ into a quantum state we normalize it by defining $\tilde{T}$:
\begin{align}
    \tilde{T} = \frac{T}{\lnorm T \rnorm_{2}},
\end{align}
with the Hilbert-Schmidt norm of the tensor given as $\lnorm T \rnorm_{2}^2=\sum_{i_1 , \dots , i_p} T_{i_1 , \dots , i_p}^2$. By linearity, it follows:
\begin{align}
    \injnorm{\tilde{T}} = \frac{\injnorm{T}}{\lnorm T \rnorm_{2}} .
\end{align}
Given a tensor $T$ we define the function $f_T: Gr(p,d) \rightarrow \R$ as:
\begin{align}
    f_T(x^{(1)}\wedge \ldots \wedge x^{(p)})=\frac{\langle T, x^{(1)}\wedge \ldots \wedge x^{(p)}\rangle}{\lnorm x^{(1)}\wedge \ldots \wedge x^{(p)}\rnorm_2}.
\end{align}
It is clear that the maximum of the function $f_T$ is the injective norm of the tensor $T$. 
In the asymptotic regime, it is more convenient to instead consider the normalized version:
\begin{align} \label{eq:definition_f}
    f_T(x^{(1)}\wedge \ldots \wedge x^{(p)})= \frac{1}{\sqrt{p(d-p)}}\frac{\langle T, x^{(1)}\wedge \ldots \wedge x^{(p)}\rangle}{\lnorm x^{(1)}\wedge \ldots \wedge x^{(p)}\rnorm_2}.  
\end{align}
 \\

\begin{definition}
    Let $T$ be a skew-symmetric complex p-tensor. A function $g_T: (\C^d)^p \rightarrow \R$ can be defined:
    \begin{align}
        g_T(x^1,\dots, x^p)= \frac{1}{\sqrt{2p(d-p)}} \frac{\Re{\langle T, x^{(1)} \wedge\dots \wedge x^{(p)}\rangle}}{{\lnorm x^{(1)}\wedge \ldots \wedge x^{(p)}\rnorm_2}}.
    \end{align}
\end{definition}
\noindent  Lemma \ref{eq-comp-def} shows that  the function $g_T$ can be used in place of the function $    f_T(x^{(1)}\wedge \ldots \wedge x^{(p)})$ in the complex setting.

\begin{lemma} \label{lem: Frob. norm concen.}In the real case, the squared Hilbert-Schmidt norm of a skew-symmetric tensor $T$ is distributed as $p!$ times a $\chi^2$ random variable with ${d \choose p}$ degrees of freedom. In the complex case, it has the same law as $\frac{p!}{2}$ times a $\chi^2$-distribution with $2 {d \choose p}$ degrees of freedom. For any $x>0$, the following concentration inequalities hold for the real and complex case
\begin{align}
    \label{Frob-estimate}
    real \; case: \quad \P\left(\lnorm T \rnorm_{2}^2 \leq \frac{d!}{(d-p)!} - 2 \sqrt{\frac{{p!}d!}{(d-p)! }}x\right) \leq \exp(-x^2) \\ \nonumber
    complex \; case: \quad \P\left(\lnorm T \rnorm_{2}^2 \leq \frac{d!}{(d-p)!} -  \sqrt{2\frac{{p!}d!}{(d-p)! }}x\right) \leq \exp(-x^2)
\end{align}
In particular the mean $\frac{d!}{(d-p)!}$ behaves asymptotically as follows for the two cases of (fixed $p$, $d\rightarrow \infty$) and (fixed $\alpha\in(0,1),p=\lfloor \alpha d\rfloor$, $d\rightarrow \infty$):
\begin{align}
    \label{frob_concentrate}
    &(\text{fixed} \; p,\;d\rightarrow \infty) : \quad \E(\lnorm T \rnorm_{2}^2) = d^{p} {\left(1+o\left(\frac1d\right)\right)}\\  \nonumber
    &(\text{fixed} \; p=\lfloor \alpha d\rfloor ,\;d\rightarrow \infty) : \quad \E(\lnorm T \rnorm_{2}^2) =  d^{\alpha d}e^{-(\alpha +(1-\alpha)\log(1-\alpha))d} {\left(1+o\left(\frac1d\right)\right)}
\end{align}
\end{lemma}

\begin{proof}
    In the real case, the Hilbert-Schmidt norm of a random skew-symmetric tensor is  a random variable  given as: 
    \begin{align}
        \lnorm T \rnorm^2_{2} = p! \sum_{i_1<\dots<i_p} {T}_{i_1,\dots,i_p}^2.
    \end{align}
    We can define the sum as  a random variable $\mathcal{S}=\sum_{i_1<\dots<i_p} {T}_{i_1,\dots,i_p}^2$, which is a $\chi^2$ distribution with $d \choose p$ degrees of freedom. After this, equation \eqref{Frob-estimate} and the identity $\E(\lnorm T \rnorm^2_F)= \frac{d!}{(d-p)!}$  follow directly from \cite[lemma 1]{laurent2000adaptive}. Using Stirling's formula in the  large-$d$ limit of the fixed-$p$ regime, the mean becomes:
    \begin{align}
        \lim_{d \rightarrow \infty} \E(\lnorm T \rnorm^2_{2}) = d^p\left(1+o\left(\frac1d\right)\right).
    \end{align}
   Furthermore, for the double-scaling behavior of $\E(\lnorm T \rnorm_{2}^2)=\E(\lnorm T\rnorm_{2}^2)=\prod_{i=0}^{p-1}(d-i)$, one obtains,
    \begin{align*}
    \log(\E(\lnorm T \rnorm_{2}^2))&=\sum_{i=0}^{p-1}\log(d-i)\le \int_{-1}^{p-1}\log(d-x)\mathrm{d}x\\
    &=\alpha d\int_0^1[\log d + \log(1+1/d-\alpha v)]\mathrm{d}v =\alpha d \log d+d\int_{1+1/d-\alpha}^{1+1/d}\log x \,\mathrm{d}x\\
    &=\left[\alpha d\log d-d(\alpha+(1-\alpha)\log(1-\alpha))\right](1+o(1)).
    \end{align*}
    A lower bound comes 
    \begin{align*}
        \sum_{i=0}^{p-1}\log(d-i)&\ge \int_0^p\log(d-x)\mathrm{d}x\\
        &=\alpha d\log d -d[\alpha+(1-\alpha) \log(1-\alpha)].
    \end{align*}
    Therefore, we obtain the asymptotic behavior, 
    \begin{equation}
        \E(\lnorm T \rnorm_{2}^2)= d^{\alpha d}e^{-(\alpha +(1-\alpha)\log(1-\alpha))d}\left(1+o\left(\frac1d\right)\right),
    \end{equation}
which proves equation \eqref{frob_concentrate} for the real case.\\
In the complex case, we have $\E(\lnorm T \rnorm_{2}^2) = \frac{p!}{2}\mathcal{S}_\C$ with $\mathcal{S}_\C$ being a $\chi^2$ distribution with $2 {d \choose p}$ degrees of freedom. The rest of the calculation is similar and leads to the same result \eqref{frob_concentrate}. Applying \cite[lemma 1]{laurent2000adaptive} yields the complex part of \eqref{Frob-estimate}. 
\end{proof}
\subsubsection{Properties of injective norm of fermions and their geometric entanglement}
The injective norm naturally leads to the definition of the geometric measure of entanglement (GME) as:
\begin{equation}
    \text{GME}(\ket{\psi_f})=-2\log(\lnorm \ket{\psi_f}\rnorm_{\text{inj}}).
\end{equation}
We now recall the definition of the Hodge dual
        \begin{align}\label{def:Hodge_dual}
    *: \bigwedge^p \mathrm{R}^d \rightarrow \bigwedge^{d-p} \mathrm{R}^d, \quad \ket{\psi}=\sum_{\sigma \in ([d])_p} \psi_\sigma e_{\sigma_1} \wedge ... \wedge e_{\sigma_p} \mapsto *\ket{\psi}=\sum_{\sigma \in ([d])_p} \psi_\sigma e_{\sigma^c_1} \wedge ... \wedge e_{\sigma^c_{d-p}}  
\end{align}
where $\sigma$ is an injective (i.e. non-repeating) $p$-tuple from the set $\{1, \ldots,d \}$. Since the order of the components of $\sigma $ is important, we only consider, without loss of generality, $\sigma$ such that $\sigma_1<\ldots< \sigma_p$. Furthermore, we assign to each $p$-tuple a $(d-p)$-tuple $\sigma^c$ and we require it to have an ordering such that $\{\sigma_1,\ldots,\sigma_p,\sigma^c_1\,\ldots,\sigma^c_{d-p}\}$ is an even permutation of $\{1,\ldots,d\}$. 
 
\begin{example} Let $\sigma =(1,3)$ and $\sigma \in {[4] \choose 2} $ then  $\sigma^c$ is ordered as $\sigma^c=(4,2)$ since $(1,3,4,2)$ is an even permutation of $(1,2,3,4)$.
\end{example}
\noindent Given a state of $p$ fermions $\psi_f$ the Hodge dual maps it to a state of $d-p$ fermions $\ket{\psi_h}=*\ket{\psi_f}$, by swapping the filled and empty dimensions.
\begin{proposition}[Geometric entanglement particle-hole duality]\label{prop:particle-hole-duality}
For all $\ket{\psi_f}\in \bigwedge^p\R^d$ 
\begin{equation}
 \injnorm{\ket{\psi_f}}=\injnorm{\ket{\psi_h}},
\end{equation}
which translates at the level of the Geometric Measure of Entanglement
\begin{equation}
    \text{GME}(\ket{\psi_f})=\text{GME}(\ket{\psi_h}).
\end{equation}
\end{proposition}
\begin{proof}
      The Hodge dual preserves the Hermitian inner product up to a sign. Let $\ket{\psi_f}\in \bigwedge^p \C^d$ be a fermionic state and let $x_1\wedge\ldots\wedge x_p$ be chosen to be the set of vectors that maximize the function given in the definition of the injective norm of $\ket{\psi_f}$. \\ Denote $\ket{\psi_h}:=*\ket{\psi_f}$ the hole state dual to $\ket{\psi_f}$ in \eqref{eq:Injective_norm}.  Then one has 
      \begin{equation}
          \langle \psi_f, x_1\wedge\ldots \wedge x_p\rangle=\langle \psi_h,*(x_1\wedge \ldots \wedge x_p) \rangle.
      \end{equation}
    Denote $v=[x_1\wedge\ldots\wedge x_p]\in \text{Gr}(p,d)$. Assume there exists $\tilde {v} \in \text{Gr}(d-p,d), \, \tilde{v}\neq v$, represented by $y_1\wedge \ldots \wedge y_{d-p}\in \bigwedge^{d-p} \C^d$, so that $\lvert\langle \psi_f,x_1\wedge \ldots \wedge x_p\rangle\rvert < \lvert \langle \psi_h, y_1\wedge \ldots \wedge y_{d-p}\rangle\rvert$ then by isometry of the Hodge dual, we have
    $$\lvert \langle \psi_f, x_1\wedge\ldots \wedge x_p\rangle\rvert<\lvert \langle \psi_f, *(y_1\wedge \ldots \wedge y_{d-p})\rangle\rvert,$$
    where $[*(y_1\wedge \ldots \wedge y_{d-p})] \in \text{Gr}(p,d)$ which leads to a contradiction.   
\end{proof}
We now examine how the random function $f_T$ transforms under Hodge duality in a neighborhood of the "\textit{north pole}" of the Grassmannian. \footnote{For the exact definition of this reference point see Appendix \ref{AppA}.}\\ For this case, we first look at the transformation of the neighborhood itself under the action of Hodge star operator, that is $U \rightarrow *U$ with $U \subset Gr(p,d)$ and $*U \subset Gr(d-p,d)$. For a small enough $t$ the coordinates of a basis vector in the chart described in the Appendix \ref{AppA} can be displayed as:
\begin{align*}
M= \begin{pmatrix}
    I_p \\ 
    t E^{(d-p) \times p}_{ij}
\end{pmatrix},
\end{align*}
where $E^{(d-p) \times p}_{ij}$ is the $(d-p) \times p$-matrix with $1$ in the $ij$-th entry and zeros everywhere else. The point described by these coordinates corresponds to the space spanned by:
\begin{align*}
    e_1 \wedge ... \wedge (e_j+t e_{p+i}) \wedge ... \wedge e_p.
\end{align*}
This wedge product under the action of the Hodge operator is transformed to
\begin{align*}
    -e_{p+1} \wedge...\wedge e_{p+i-1} \wedge (e_{p+i}+te_j) \wedge e_{p+i+1} \wedge ...\wedge e_d ,
\end{align*}
which in the matrix form is presented as:
\begin{align*}
 M= \begin{pmatrix}
    t E^{p \times (d-p)}_{ji} \\ 
    I_{(d-p)}
\end{pmatrix}.
\end{align*}
Using the fact that Hodge operator is linear, it can be shown that the matrix $B \in \R^{(d-p) \times p} $ representing a point on the chart of $U$ is transformed into $B^T \in \R^{p \times (d-p)}$ on the chart of $*U$ under the pushforward of the Hodge duality and therefore the correlations of the functions and their derivatives as we calculated in Appendix \ref{AppB} are preserved under the action of the Hodge operator.  
\begin{definition}
A $\textup{BHGOE}(d,p,\sigma^2)$ random matrix is a $pd \times pd$ real-symmetric random matrix, thought of as being partitioned into $p^2$ blocks of size $d \times d$, with the diagonal blocks set to zero, and the remaining entries are independent up to symmetry, each distributed according to a normal distribution $\mathcal{N}(0,\sigma^2)$. In other words, such a matrix is a GOE matrix with the diagonal blocks zeroed out.
\end{definition}
For all $m,k \in \mathbb{Z}_+$, we define the partial transpose $\ammaG: \text{Mat}_{k}(\C)\otimes \text{Mat}_{m}(\C)\rightarrow \text{Mat}_{k}(\C)\otimes \text{Mat}_{m}(\C)$ as the map $\text{Id}\otimes t$ where $t:\text{Mat}_{m}(\C)\rightarrow \text{Mat}_{m}(\C)$ denotes the transpose of a matrix. 
Let $N=p(d-p)$ and let the family $E_{ij}^{(p)}$ be the standard matrix basis elements of $\text{Mat}_p(\C)$ (that is, $E_{ij}^{(p)}$ is the $p\times p$ matrix with a one at position $(i,j)$ and zeros everywhere else).
\begin{definition}[Block Hollow Gaussian Antisymmetric Ensemble]\label{def:BHGAE}
Let $d,p\in \mathbb{Z}_{>0}$ and let $A_N$ be a $N \times N$ $\text{BHGOE}(d,p,\sigma^2)$ random matrix. Then,
$$A_N=\sum_{1\le i < j \le p} (E_{ij}^{(p)}\otimes G_{i,j}+E^{(p)}_{ji}\otimes G_{i,j}^t),$$
where the $G_{i,j}$ are independent $d\times d$ random matrices whose elements are independent and distributed according to a normal distribution $\mathcal{N}(0,\sigma^2).$

We define $W_N$ to a be an $\text{BHGAE}(d,p,\sigma^2)$ as being the partial antisymmetrization of $A_N$
$$W_N:=\frac{1}{\sqrt{2}}(A_N-A_N^{\ammaG}),$$
where $A_N^{\ammaG}$ is the partial transpose of $A_N$, so that formally,
$$A_N^{\ammaG}=\sum_{1\le i < j \le p} (E_{ij}^{(p)}\otimes G_{i,j}^t+E^{(p)}_{ji}\otimes G_{i,j}).$$
\end{definition}
By construction, $W_N$ consists of $p^2$ blocks, with vanishing diagonal blocks. We index the elements of $W_N$ by pairs of doublets $(i,\alpha),(j,\beta)$, where $i,j$ determine the block and $\alpha, \beta$ the line and column inside a block. A given non trivial block $(i,j)$ is a skew-symmetric random matrix.  Its independent elements $(W_N)_{(i,\alpha),(j,\beta)}$ for $\alpha<\beta$ are $\mathcal{N}(0,\sigma^2)$ random variables, while the elements on the diagonal of a block vanish. \\ 
We use the work \cite{bandeira2023matrix} to obtain the following. The proof is very similar to what can be found in \cite[Lemma 2.3 \& Corollary 2.4]{dartois2024injective}
\begin{proposition}[Operator norm upper bound]\label{prop:operator_norm_bound}
Let $W_N$ be a $\text{BHGAE}(d-p,p,\frac{1}{p(d-p)})$ matrix. There exists a universal constant $C>0$ such that for all $t\ge 0$
\begin{equation}
   \P\left(\lnorm W_N \rnorm > 2\sqrt{\frac{p-1}{p}}+\frac{2C}{\sqrt{p(d-p)}}\sqrt{\frac{p-1}{p}\frac{d-p-1}{d-p}}(\log(p(d-p)))^{3/4}+Ct\right)\le e^{-t^2}. 
\end{equation}
This implies that for any fixed $\epsilon,\delta$ and $p$ fixed, there exists $d_0(\epsilon,\delta)$ so that
\begin{equation}
    \P\left(\lnorm W_N\rnorm> 2\sqrt{\frac{p-1}{p}}+\epsilon\right)\le e^{-N^{1-\delta}} \quad \forall d\ge d_0(\epsilon,\delta).
\end{equation}
In particular for $p \rightarrow \infty$
\begin{equation}
    \P\left( \lnorm W_N \rnorm > 2+\epsilon \right) \leq e^{-N^{1-\delta}} \quad \forall d \geq d_0(\epsilon,\delta).
\end{equation}
\end{proposition}
{
\begin{proof}
    We rely on the results \cite[corollary 2.2]{bandeira2023matrix} and recall them here for the computations to come. Let the parameters $\sigma$ and $v$ for a matrix $X$ be defined as:\\
\begin{align*}
    \sigma(X)^2=\lnorm\E(X^2)\rnorm \\
    v(X)^2=\lnorm\Cov(X)\rnorm,
\end{align*}\
then
    \begin{equation}
       p(d-p) \sigma(W_N)^2=(d-p-1)\left\lvert\!\left\lvert \sum_{1\le i<j\le p} (E_{ii}^{(p)}+E_{jj}^{(p)})\otimes I_{d-p}\right\rvert\!\right\rvert_{op}=(p-1)(d-p-1). 
    \end{equation}
    The covariance operator norm is given by
    \begin{equation}
        p(d-p)v(W_N)^2:=\left\lvert\!\left\lvert \text{Cov}(W_N)\right\rvert\!\right\rvert_{op}=2\left\lvert\!\left\lvert P\otimes Q\right\rvert\!\right\rvert_{op}, \label{projector_covariance}
    \end{equation}
    where $P:\C^p\rightarrow \C^p$ and $Q:\C^{d-p}\rightarrow \C^{d-p}$ are projectors, so that $P\otimes Q:\C^{p\times (d-p)}\rightarrow \C^{p\times (d-p)}$. The explicit form of $P$ and $Q$ is given below
    \begin{equation}
        P=\sum_{1\le i<j\le p}\alpha_{ij}^{(p)}\otimes \left(\alpha_{ij}^{(p)}\right)^*, \quad Q=\sum_{1\le k<l\le d-p}\alpha_{kl}^{(d-p)}\otimes \left(\alpha_{kl}^{(d-p)}\right)^*, 
    \end{equation}
    where $\alpha_{ij}^{(p)}=(E_{ij}^{(p)}-E_{ji}^{(p)})$ and $(\cdot)^*$ denotes the dual induced by the Frobenius scalar product of matrices.
    Using \cite[corollary 2.2]{bandeira2023matrix}, one has, using the notations introduced in \cite{bandeira2023matrix}, and denoting $W_{\text{free}}$ the non commutative model for $W_N$ (playing the role of $X_{\text{free}}$ for $X_N$ in \cite{bandeira2023matrix}) 
    $$\P(\lnorm W_N\rnorm_{op}\ge \lnorm W_{\text{free}}\rnorm+C \tilde{v}(W_N)(\log N)^{3/4}+Ct\sigma_*(W_N))\le e^{-t^2},$$
    for all $t\ge 0$ and $C$ a universal constant. We recall that (see \cite[page 10]{bandeira2023matrix})
    \begin{align}
        &\tilde{v}(W_N)=\sigma(W_N)v(W_N)=\frac{2}{\sqrt{p(d-p)}}\sqrt{\frac{p-1}{p}\frac{d-p-1}{d-p}} \\ \nonumber & \sigma_*(W_N)\le \sigma(W_N)=\sqrt{\frac{p-1}{p}\frac{d-p-1}{d-p}}.
    \end{align}
    According to the estimate \cite[lemma 2.5]{bandeira2023matrix}, whose first version is attributed to Pisier \cite[page 208]{Pisier_2003}, $\lnorm W_{\text{free}}\rnorm_{\text{op}}\le 2\sigma(W_N)\le2\sqrt{\frac{p-1}{p}}.$
     This leads to the following explicit bound for the operator norm of $W_N$
     \begin{equation}
         \P\left(\lnorm W_N \rnorm > 2\sqrt{\frac{p-1}{p}}+\frac{2C}{p(d-p)}\sqrt{\frac{p-1}{p}\frac{d-p-1}{d-p}}+Ct{ \sqrt{\frac{p-1}{p}\frac{d-p-1}{d-p}}}\right)\le e^{-t^2}.
     \end{equation}
\end{proof}}
\begin{definition}[Complex BHGAE] \label{def:complex BHGAE}
We define a complex block hollow Gaussian antisymmetric ensemble, $\text{cBHGAE}(p,d,\sigma^2)$, as given by the following block form:
$$V_N=\begin{pmatrix}
\begin{array}{c|c}
A & B \\ \hline
B & -A
\end{array}
\end{pmatrix}$$
where $A$ and $B$ are  independent $\text{BHGAE}(p,d,\sigma^2)$ matrices.
\end{definition}
\begin{proposition} \label{prop:complex_operator_bound}    
Let $V_N$ be a $\text{cBHGAE}(d-p,p,\frac{1}{2p(d-p)})$.
There exists a universal constant $C>0$ such that for all $t\ge 0$
\begin{equation}
   \P\left(\lnorm V_N \rnorm > 2\sqrt{\frac{p-1}{p}}+\frac{2C}{\sqrt{p(d-p)}}\sqrt{\frac{p-1}{p}\frac{d-p-1}{d-p}}(\log(p(d-p)))^{3/4}+Ct\right)\le e^{-t^2}. 
\end{equation}
Similarly, this implies that for any fixed $\epsilon,\delta$ and $p$ fixed, there exists $d_0(\epsilon,\delta)$ so that
\begin{equation}
    \P\left(\lnorm V_N \rnorm> 2\sqrt{\frac{p-1}{p}}+\epsilon\right)\le e^{-N^{1-\delta}} \quad \forall d\ge d_0(\epsilon,\delta).
\end{equation}
In particular, for $p \rightarrow \infty$
\begin{equation}
    \P\left( \lnorm V_N \rnorm > 2+\epsilon \right) \leq e^{-N^{1-\delta}} \quad \forall d \geq d_0(\epsilon,\delta)
\end{equation}
\end{proposition}
\proof
The proof follows the one of  proposition \ref{prop:operator_norm_bound}. In particular the cBHGAE matrix can be represented as 
$$V_N=(E_{11}-E_{22})\otimes A + (E_{21}+E_{12})\otimes B. $$
Thus:
\begin{equation}
    \sigma(V_N)^2 = 2 \frac{1}{2 p (d-p)} (p-1)(d-p-1).
\end{equation}
Furthermore, the structure of $\Cov(V_N)$ is given by:
$$\Cov(V_N)=\begin{pmatrix}
\begin{array}{c|c}
\begin{matrix}
\Cov(A) & -\Cov(A) \\
-\Cov(A) & \Cov(A)
\end{matrix} & 0 \\ \hline
0 & 
\begin{matrix}
\Cov(B) & \Cov(B) \\
\Cov(B) & \Cov(B)
\end{matrix}
\end{array}
\end{pmatrix}$$
where $\Cov(A)$ and $\Cov(B)$ are given by equation \ref{projector_covariance}. Hence we have for $\Cov(V_N)$:
\begin{equation}
    \left\lvert\!\left\lvert \Cov(V_N)\right\rvert\!\right\rvert_{op} = 4\sigma^2
\end{equation}
which coincides with the real case, i.e. for $\text{BHGAE}(p,d-p,\frac{1}{p(d-p)})$, if divided by the 2 coming from the difference in the normalization factor. Hence, the rest follows exactly as in the proof of \ref{prop:operator_norm_bound}.
\endproof
\noindent Finally, let $\Sigma^{(p)}_\R(u): \R \rightarrow \R$ be  defined as: 
$$\Sigma^{(p)}_\R(u)=\frac{1+\log p}{2}-u^2/2+\Omega\left(u\sqrt{\frac{p-1}{p}}\right){+\frac12\log(\frac{p-1}{p})},$$
with the log-potential of the measure $\mu(x)= \frac{1}{2 \pi} \sqrt{(4-x^2)_+}$ defined by:
    \begin{align}
        \Omega_\mu(u) = \int_{\R} \log|u-\lambda| \dd \mu(\lambda).
    \end{align}
\noindent Then it has been shown:
\begin{lemma}\label{lem:annealed_complexity_real}
$\Sigma_\R^{(p)}(u)$ is strictly decreasing and has a unique solution $E_0(p)$ to $\Sigma_\R^{(p)}(E_0(p))=0$ on the positive real line.
\end{lemma}
\begin{proof}
     $\Sigma_\R^{(p)}(u)$ already appears in the literature, and similar lemma have been shown in for instance \cite{dartois2024injective, mckenna2024complexity,auffinger2013random}. We refer to those for the proof.
\end{proof}

\section{Real Case}
\noindent {\bf Kac-Rice formula:}
An extension of the Kac-Rice formula, obtained in the same way as in \cite{dartois2024injective}, allows us to prove 
\begin{proposition}\label{prop:upper-KR}
    Let $f_N$ be the Gaussian process described above. Then, for all $B$ a Borel set of $\R$,
    \begin{multline}\E(\Crt_{f_N}(B))\le \pi^{\frac12p(d-p)}\frac{\prod_{i=1}^p \Gamma(i/2)\prod_{i=1}^{d-p}\Gamma(i/2)}{\prod_{i=1}^d\Gamma(i/2)}\left(\sqrt{\frac{p(d-p)}{2\pi}}\right)^{p(d-p)}\\
    \times \int_B\sqrt{\frac{p(d-p)}{2\pi}}\diff u \, e^{-\frac{p(d-p)}{2}u^2}\E_{\text{BHGAE}}\left( \lvert\det(u-W_N) \rvert \right),
    \end{multline}
    where $W_N$ is a $\text{BHGAE}(d-p,p,\frac1{p(d-p)})$ random matrix as in definition \ref{def:BHGAE} and $\E_{\text{BHGAE}}$ denotes the average with respect to the randomness of $W_N$.
\end{proposition}
\begin{proof}
The proof of the proposition is very similar to the one given in the appendix of \cite{dartois2024injective}. The only difference arises from the particularities of the geometry of the Grassmann manifold. Some considerations regarding the specific geometry of the Grassmannian and its volume are shown in the appendix \ref{AppA} and in particular lemma \ref{Volumeapp}. 
\end{proof}
\subsection{Real, fixed $p$, $d \rightarrow \infty$}
\begin{theorem}[Main theorem - real case]
Let $p\ge 3$ be a fixed integer and let $T\in \bigwedge^p\R^d$ be a random skew-symmetric tensor with i.i.d. $\mathcal{N}_\R(0,1)$ entries. Then: 
$$\limsup_{d\rightarrow \infty}\frac1{p(d-p)}\log \P\left(\frac1{\sqrt{p(d-p)}}\injnorm{T}> E_0(p)+\epsilon\right)<0.$$
Defining the normalized state $\ket{\psi_f}:=\frac{T}{\lnorm T\rnorm_{2}}$. Then:
$$\limsup_{d\rightarrow \infty}\frac1{p(d-p)}\log \P\left(\frac{\injnorm{\ket{\psi_f}}}{\sqrt{p(d-p)}}>\frac1{d^{\frac{p}{2}}}(E_0(p)+\epsilon)\right)<0.$$
\end{theorem}
\begin{proof}
   We obtain this theorem by applying \cite[Theorem 4.1]{arous2022exponential} to estimate the upper bound on the annealed complexity $\log \E(\Crt_{f_T}(B))$ of the Gaussian process $f_T$ expressed in proposition \ref{prop:assumptions_check} for $B\subset(E_0(p),+\infty)$. \\
   
   The assumptions of \cite[Theorem 4.1]{arous2022exponential} are checked in proposition \ref{prop:assumptions_check}.\\

\noindent The upper bound on the injective norm follows from the application of the Markov inequality to the probability $\P\left(\Crt_{f_T}((E_0(p)+\epsilon, \infty))\ge 1\right).$ Formally, 
\begin{align*}
    \frac1{p(d-p)}\log\P\left(\frac1{\sqrt{p(d-p)}}\injnorm{T}> E_0(p)+\epsilon\right)&=\frac1{p(d-p)}\log \P\left(\Crt_{f_T}((E_0(p)+\epsilon, +\infty))\ge 1\right)\\
    &\le \frac1{p(d-p)}\log \E\left(\Crt_{f_T}((E_0(p)+\epsilon, +\infty))\right).
\end{align*}
According to lemma \ref{lem:annealed_complexity_real} and \cite[theorem 4.1]{arous2022exponential} one finds { for all $\epsilon>0$}
$$\frac1{p(d-p)}\log\P\left(\frac1{\sqrt{p(d-p)}}\injnorm{T}> E_0(p)+\epsilon\right)\le \sup_{u \in (E_0(p)+\epsilon, \infty)} \Sigma_\R(u)\le 0.$$

\noindent The normalized version follows then as a result of lemma \ref{lem: Frob. norm concen.}.
\end{proof}

  \begin{lemma}\label{lem:asymptotics_constant_KR}
     Let  $$K(p,d):=\pi^{\frac12p(d-p)}\frac{\prod_{i=1}^p \Gamma(i/2)\prod_{i=1}^{d-p}\Gamma(i/2)}{\prod_{i=1}^d\Gamma(i/2)}\left(\sqrt{\frac{p(d-p)}{2\pi}}\right)^{p(d-p)+1},$$
     then 
     \begin{equation}
         \lim_{d\rightarrow\infty}\frac1{p(d-p)}\log K(p,d)=\frac{\log p +1}{2}.
     \end{equation}
  \end{lemma}
  \begin{proof}
 Noting that
      \begin{multline}
      \frac1{p(d-p)}\log K(p,d)=\frac1{p(d-p)}\Bigl(\frac12 p(d-p)\log \pi +\frac12(p(d-p)+1)\log p(d-p)\\
      - \frac12(p(d-p)+1)\log2\pi+\sum_{i=1}^{d-p}\log \Gamma(i/2)+\sum_{i=1}^p\log\Gamma(i/2)-\sum_{i=1}^d \log\Gamma(i/2) \Bigr).\\
      \end{multline}
      The only needed remarks are that for $p$ finite
      \begin{equation}
          \lim_{d\rightarrow \infty}\frac1{p(d-p)}\sum_{i=1}^p\log \Gamma(i/2)=0, \label{limp0}
      \end{equation}
      and bounding the sum:
      \begin{equation}
         \frac1{d-p}\log\Gamma\left(\frac{d-p}{2}\right)\le\frac1{p(d-p)} \left(\sum_{i=1}^d \log \Gamma(i/2)-\sum_{i=1}^{d-p}\log \Gamma(i/2)\right)\le \frac1{d-p}\log\Gamma\left(\frac{d}{2}\right). 
      \end{equation}
      Using Stirling's approximation on the upper and lower bound, we conclude that 
      \begin{equation}
         \lim_{d\rightarrow\infty}\frac1{p(d-p)}\log K(p,d)=\frac{\log p +1}{2}.
     \end{equation}
     This finishes the proof.
  \end{proof}

In particular, we recall the necessary assumptions (see \cite[theorem 1.2]{arous2022exponential} and its assumptions) that we have to check:
\begin{enumerate}
    \item \label{enum:prop_Wass} Control of the Wasserstein distance: Let $\mu_{H_N(u)}=\frac1{N}\sum_{i=1}^{N}\delta_{\lambda_i(H_N(u))}$ be the empirical spectral measure of the Hessian. According to \cite[Theorem 4.1 \& Theorem 1.2]{arous2022exponential}, one needs to show that $\exists \kappa>0$ such that $W_1(\E(\mu_{H_N(u)}),\mu_p(u))<N^{-\kappa}$.
    \item \label{enum:prop_Lip-Trace}Concentration of Lipschitz trace: One needs to prove that for every Lipschitz  function $f:\R\rightarrow \R$, $\exists C,c>0$ such that , 
    $$\, \P\left(\left\lvert\frac1N\Tr(f(H_N(u)))-\frac1N\E(\Tr(f(H_N(u)))) \right\rvert\ge \delta \right)\le Ce^{-c\frac{N^2 \delta^2}{\lnorm f\rnorm^2_{\text{Lip}} }}.$$
    \item \label{enum:prop_gap}Gap assumption:
    $$\forall \epsilon>0, \, \lim_{N\rightarrow \infty}\P\left(\text{Spec}(H_N(u))\cap [-e^{-N^\epsilon},e^{-N^\epsilon}]\right)=0,$$
    and as was realized  through lemma 3.11 in \cite{dartois2024injective}, one only needs this property to be true for $\lvert u \rvert$ large enough.
\end{enumerate}
We also need specifically from \cite[theorem 4.1]{arous2022exponential}
\begin{enumerate}
\setcounter{enumi}{3}
    \item \label{enum:prop_bound_determinant} There exists $C>0$ such that $\E\left(\lvert \det(H_N(u))\rvert \right)\le (C\max(\lvert u\rvert ,1))^N$ together with the map $u\mapsto H_N(u)$ being entrywise continuous (in our case it is just a translation of the diagonal elements by $u$).
\end{enumerate}
In fact, letting $\mathcal{A}:=(-\infty, -E_0(p))\cup (E_0(p),\infty)$, one shows, 
\begin{proposition}\label{prop:assumptions_check}
Let $H_N(u)=u-W_N$, where $W_N$ is a BHGAE. Then properties \ref{enum:prop_Wass}, \ref{enum:prop_Lip-Trace}, \ref{enum:prop_bound_determinant} are satisfied for all $u\in \R.$  Assuming $u\in \mathcal{K}\subset \mathcal{A}$ closed, property \ref{enum:prop_gap} is also satisfied.  
\end{proposition}
\begin{proof}
The proof consists in checking that the arguments proving lemma 3.15, lemma 3.16 and lemma 3.17 in \cite{dartois2024injective} can be adapted. In fact, \textbf{property \ref{enum:prop_Lip-Trace}} , the concentration of Lipschitz trace, is obtained as a consequence of the Herbst argument, which can also here be adapted from \cite[Lemma 2.3.3 \& Theorem 2.3.5]{anderson2010introduction} noticing that the function $g:\R^{p\frac{(d-p)(d-p-1)}{2}}\rightarrow \R$ maps the vector of $p\frac{(d-p)(d-p-1)}{2}$ independent entries  of $W_N$ to $\Tr(f(W_N))$ where $f$ is the Lipschitz function given in property \ref{enum:prop_Lip-Trace}. Thus, it can be shown that $g$ is also a Lipschitz function with a Lipschitz constant of order $\sqrt{N} \lnorm f\rnorm_{Lip}$. The rest of the Herbst argument can be taken analogously as in \cite{dartois2024injective}.  The only difference in this case is the argument that $\sup_N \E[\lnorm W_N\rnorm]<\infty$ , which is shown in proposition \ref{prop:operator_norm_bound}.\\
\textbf{Property \ref{enum:prop_Wass}} is obtained in the same way, as in \cite{dartois2024injective}. To this aim we need to import the local law of \cite{alt2019location}. This is obtained as a consequence of the Matrix Dyson Equation for $W_N$. 
Define, for all $i<j$, $\beta_{ij}^{(p)}:=E_{ij}^{(p)}-E_{ji}^{(p)}$ so that $(\beta_{ij}^{(p)})^t=-\beta_{ij}^{(p)}$, and $G_{ij}$ be a  random matrix with i.i.d. standard Gaussian entries. Then a BHGAE matrix $W_N$ can be expressed as:
$$W_N={\frac1{\sqrt{2p(d-p)}}}\sum_{1\le i<j\le p}\beta_{ij}^{(p)}\otimes G_{ij}+(\beta^{(p)}_{ij})^t\otimes G_{ij}^t,$$
so that we recover the form of a Kronecker random matrix, where comparing with notations from \cite[Definition 2.1]{alt2019location}, $L=p$, $N=d-p$ and, $\gamma_{ij}^{(p)}=-\beta_{ij}^{(p)}$. \\
The components $\mathscr{S}_i$ of the operator $\mathscr{S}$ do not depend on $i$, and we have $$\mathscr{S}_i[r]={\frac{1}{2p(d-p)}}\sum_{k=1}^{d-p}\sum_{1\le i < j \le p}\beta^{(p)}_{ij}r_k(\beta^{(p)}_{ij)})^t+\gamma^{(p)}_{ij}r_k(\gamma^{(p)}_{ij})^t,$$
so that the matrix Dyson equation is, $\forall i\in \{1,\ldots, d-p\}$ 
\begin{align}
    &\Id+(z\Id+\mathscr{S}_i[m(z)])m_i(z)=0, \label{MDE1}\\
    &\Im{m_i(z)}> 0, \label{MDE2}
\end{align}
As in \cite{dartois2024injective}, let $r=(c\Id,\ldots, c\Id)$ where $\Id$ is the $(d
-p)\times(d-p)$ identity matrix, one finds
$$\mathscr{S}_i[c\Id]=\frac{p-1}{p}c\Id,$$
which is sufficient according to the proof of \cite[Lemma 3.16]{dartois2024injective} to declare that $m_i(z)=m_p(z) \Id$ with $$m_p(z)=\frac{-z+\sqrt{z^2-4\left(\frac{p-1}{p}\right)}}{2\left( \frac{p-1}{p}\right)}.$$
$m_p(z)$ is the Stieltjes transform of the the density $\rho_p(z)=\frac{p}{2\pi (p-1)}\sqrt{\left(4\left( \frac{p-1}{p}\right)-x^2\right)_+}.$
Then, using the local law \cite[(B.5)]{dartois2024injective} leads to the same bound as in the proof of  \cite[lemma 3.16]{dartois2024injective} which is sufficient to prove \ref{enum:prop_Wass}. \\
We now come to \textbf{property \ref{enum:prop_gap}}. Due to the presence of structural zeros in the Hessian it is not possible to obtain a bound on the probability of occurrence of very small singular values for any value of $u$ with available techniques in the literature. However, as was found in \cite{dartois2024injective}, the trick is to work on the event $\mathcal{E}_{\mathcal K}:=\left\{\lnorm W_N\rnorm\le 2\sqrt{\frac{p-1}{p}}+\delta_\mathcal{K}\right\}$ for $\mathcal{K}\subset \mathcal{A}$, where $\delta_\mathcal{K}=d(\mathcal{K},\bar{\mathcal{A}})$ is the distance between $\mathcal{K}$ and $\bar{\mathcal A}$. In fact, conditioned on this event, the spectrum of $\text{Hess} f_N$ is gapped away from zero (i.e. the spectrum of $\text{Hess} f_N$ is at distance at least $\delta_{\mathcal{K}}$ of $0$). This translates into the lower bound 
$$\inf_{u\in \mathcal{K}}\P(\text{Spec}(\text{Hess}\,{f_N })\cap [-e^{-N^\epsilon}, e^{N^\epsilon}]=\emptyset)\ge \P(\mathcal{E}_{\mathcal K}).$$
One then has from proposition \ref{prop:operator_norm_bound} that $\lim_{N\rightarrow \infty}\P(\mathcal{E}_{\mathcal{K}})=1.$ Proving the necessary gap property.
\textbf{Property \ref{enum:prop_bound_determinant}} is proved exactly the same way as in \cite[Lemma 4.4]{arous2024landscape} and \cite[Lemma 3.18]{dartois2024injective}.

\end{proof}
\subsection{Real, double scaling}
\begin{theorem}[Real, skew-symmetric, double scaling, $p,d \rightarrow \infty$]
For a fixed $\alpha \in (0,1)$, let $p=\lfloor \alpha d \rfloor $ and $T\in \bigwedge^p\R^d$ be a random skew-symmetric tensor with i.i.d. $\mathcal{N}_\R(0,1)$ entries and let $\gamma_\alpha(d)$ be defined as in Lemma \ref{lem: Asympgamma} then 
for every $\epsilon>0$ we have:
\begin{align}
    \limsup_ {d \rightarrow \infty} \frac{1}{d(d-p)}\log \P\left( \frac{\injnorm{T}}{\sqrt{d^2\alpha(1-\alpha)}} > \gamma_\alpha(d)+ \frac{\epsilon}{\sqrt{\log d}} \right) < 0.
\end{align}
Let $\ket{\psi_f}:=\frac{T}{\lnorm T\rnorm_{2}}$, then:
$$    \limsup_{d\rightarrow \infty}\frac1{p(d-p)}\log \P\left(\frac{\injnorm{ \ket{\psi_f}}}{\sqrt{d^2\alpha(1-\alpha)}}\ge d^{-\frac{\alpha d}{2}} e^{\frac12(\alpha + (1-\alpha)\log(1-\alpha))d} (\gamma_\alpha(d)+\epsilon)\right)<0.$$
\end{theorem}
\begin{proof}
We begin with the inequality:
\begin{align}
    &\P\left( \frac{\injnorm{T}}{d^2\alpha(1-\alpha) }>\gamma_\alpha(d)+ \frac{\epsilon}{\sqrt{\log(d)}} \right) \\ \nonumber & \leq \P\left[\text{Crt}_{f_T}([\gamma_\alpha(d)+\epsilon/\sqrt{\log(d)}, \infty)) \geq 1\right] + \P\left[\text{Crt}_{f_T}((-\infty,-\gamma_\alpha(d)-\epsilon/\sqrt{\log(d)} ]) \geq 1\right] \\ \nonumber & = 2 \P\left[\text{Crt}_{f_T}([\gamma_\alpha(d)+\epsilon/\sqrt{\log(d)}, \infty)) \geq 1\right] \leq  2 \E \left[\text{Crt}_{f_T}\left([\gamma_\alpha(d)+\epsilon/\sqrt{\log(d)}, \infty)\right) \right]
\end{align}
Furthermore, the term on the right can be bounded by:
\begin{align}
    & \limsup_{d \rightarrow \infty} \E \left[\text{Crt}_{f_T}\left([\gamma(d)+\epsilon/\sqrt{\log(d)}, \infty\right) \right] \\ \nonumber & \leq K(\alpha d,d) \int_{\gamma_\alpha(d)}^\infty e^{-N\frac{u^2}{2}} \E_{\text{BHGAE}}\left[ |\det(W_N-u)| \right] \dd u 
\end{align}
Where we have used proposition \ref{prop:upper-KR} for the upper bound.
Taking the logarithm of the right hand side and using lemmas \ref{lem: Asympgamma} and \ref{lem: LimitDeterminant} we arrive at the following expression:
\begin{align}
    & \limsup_{d \rightarrow \infty} \frac{1}{d^2 \alpha (1-\alpha)} \log \E \left[\text{Crt}_{f_T}\left([\gamma(d)+\epsilon/\sqrt{\log(d)}, \infty\right) \right] \\ \nonumber & \leq \beta(\alpha) + \limsup_{d \rightarrow \infty} \left[\frac{\log d}{2} + \Omega\left(\gamma_\alpha(d) +\epsilon/\sqrt{\log d}\right) - \frac{(\gamma_\alpha(d) +\epsilon/\sqrt{\log d})^2}{2}  \right] \leq -\epsilon 
\end{align}
with $\beta(\alpha)$ defined in lemma \ref{lem: Asympgamma}. Finally, the normalized version follows as a result of lemma \ref{lem: Frob. norm concen.}.
\end{proof}
\begin{lemma} \label{lem: asym_real_Grass_vol}
    Let $p=\lfloor\alpha d \rfloor$ and let $K(p , d)$ be defined as in \ref{lem:asymptotics_constant_KR}. Then we have:
    \begin{align}
        \lim_{d \rightarrow \infty} \frac{1}{p(d-p)} \log K(p ,d) = \frac{3}{4}+ \frac{1}{4} \frac{\alpha}{1-\alpha} \log(\alpha)&+\frac{1}{4} \frac{1-\alpha}{\alpha} \log(1-\alpha) + \frac{1}{2} \log(d) \\ \nonumber & + \frac12 \log(\alpha(1-\alpha))+ \mathcal{O}\left( \frac{\log(d)}{d}\right)
    \end{align}
\end{lemma}
\begin{proof}
The goal is to compute $\lim_{p \rightarrow \infty } \log \prod_{i=0}^p \frac{(2i)!}{(d-p+2i)!} = \lim_{p \rightarrow \infty } \sum_{i=0}^p \log \left( \frac{(2i)!}{(d-p+2i)!} \right)$. The sum is bounded from above and below by:
\begin{align}
    \int_1^{p+1} \dd x \log \left( \frac{\Gamma(x/2)}{\Gamma(\frac{d-p+x}{2})} \right) \leq \sum_{i=0}^p \log \left( \frac{\Gamma(i/2)}{\Gamma(\frac{d-p+i}{2})} \right) \leq \int_0^{p} \dd x \log \left( \frac{\Gamma(x/2)}{\Gamma(\frac{d-p+x}{2})} \right).
\end{align}
Notice that for $p,d \rightarrow \infty$ with the fixed asymptotic ratio upper and lower bound multiplied by $\frac{1}{p(d-p)}$ converge to the same value for $d \rightarrow \infty$ (the difference between them being of order $\mathcal{O}\left(\frac{\log(d)}{d^2}\right)$).
Evaluating the upper bound using the Stirling's formula, we arrive at:
{\small
\begin{align}
    &\int_0^{p} \dd x \log \left( \frac{\Gamma(x/2)}{\Gamma(\frac{d-p+x}{2})} \right)  =  d \alpha \int_0^1 \dd u \Bigg[ \frac{d \alpha }{2} u \log(\frac{d \alpha}{2} u) +\frac12 \log(\frac{1}{d \alpha u})  \\ \nonumber & - \frac{d (1- \alpha)+d \alpha u}{2} \log(\frac{d (1-\alpha)+ d \alpha u }{2})   + \frac{d(1-\alpha)}{2} - \frac12 \log(\frac{1}{d(1-\alpha) + d \alpha u}) + \mathcal{O}\left(\frac 1{d\alpha u+1}\right) \Bigg] .
\end{align}
}
Evaluating the integral, we arrive at:
\begin{align}
    &\int_0^{p} \dd x \log \left( \frac{\Gamma(x/2)}{\Gamma(\frac{d-p+x}{2})} \right)  =(d \alpha) \Bigg\{ -\frac{1}{8}(d \alpha) \left[ 1+ \log(4)-2\log(d\alpha) \right] + \frac12 \left[ 1- \log(d \alpha)\right]  \frac{d(1-\alpha)}{2}\\ & \nonumber+ \frac{1}{8 d \alpha} \left[ (d \alpha)(2d - d \alpha)(1+\log(4)+2 d^2(1-\alpha)^2 \log(d(1-\alpha)) \log(d(1-\alpha)) - 2 d^2 \log(d) \right]  \\ &  \nonumber + \frac{1}{d(1-\alpha)} \left[ d \log(d) - d \alpha- d(1-\alpha)\log(d(1-\alpha)) \right] \Bigg\} + \mathcal{O}\left(\frac 1{d\alpha u+1}\right).
\end{align}
After dividing by $\frac{1}{p(d-p)}= \frac{1}{d^2\alpha(1-\alpha)}$ the expression can be given as:
\begin{align}
    &\frac{1}{d^2 \alpha (1-\alpha)} \log K(\alpha d , d) = - \frac{1}{8} \frac{\alpha}{1-\alpha}\left( 1 + \log(4)- 2 \log(d \alpha)\right)+ \frac{1}{2d(1-\alpha)}(1 -\log(d \alpha)) \\ \nonumber + \frac{1}{8} \bigg[ \frac{2-\alpha}{1-\alpha}&(1+\log(4))+2 \frac{1-\alpha}{\alpha} \log(d(1-\alpha)) - \left( 2 \frac{1}{\alpha(1-\alpha)}+\frac{4}{d \alpha (1-\alpha)}+\frac{2}{d^2 \alpha(1-\alpha)}\log(d)\right) \bigg]\\  \nonumber +\frac{1}{d \alpha} \bigg[ \frac{1}{1-\alpha}&\log(d)- \frac{\alpha}{1-\alpha}-\log(d(1-\alpha)) \bigg] +\frac12 + \log(d) + \frac{\log(\alpha(1-\alpha))}{2} -\frac{\log(2)}{2} +\left(\frac{\log(d)}{d}\right) \\ \nonumber  = \frac18 \frac{- \alpha}{1- \alpha} &[1+\log(4)-2\log(\alpha)] + \frac{1}{4} \frac{\alpha}{1-\alpha} \log(d) + \frac18 \frac{2-\alpha}{1-\alpha} (1+\log(4))\\  \nonumber + \frac{1}{4} \frac{1-\alpha}{\alpha} & \log(1-\alpha) + \frac14 
    \frac{1-\alpha}{\alpha} \log(d) - \frac14 \frac{1}{\alpha(1-\alpha)} \log(d)+\frac12 + \mathcal{O}\left(\frac{\log(d)}{d}\right) \\ \nonumber  = 
    \frac34 + \frac14 & \frac{\alpha}{1-\alpha} \log(\alpha)+ \frac14 \frac{1-\alpha}{\alpha} \log(1-\alpha) +\frac{1}{2} \log(d)+ \frac12 \log(\alpha(1-\alpha))+ \mathcal{O}\left(\frac{\log(d)}{d}\right).
\end{align}
The lower bound also converges to the same value as mentioned above. 

\begin{figure}[h!]
    \centering
    \includegraphics[scale=0.6]{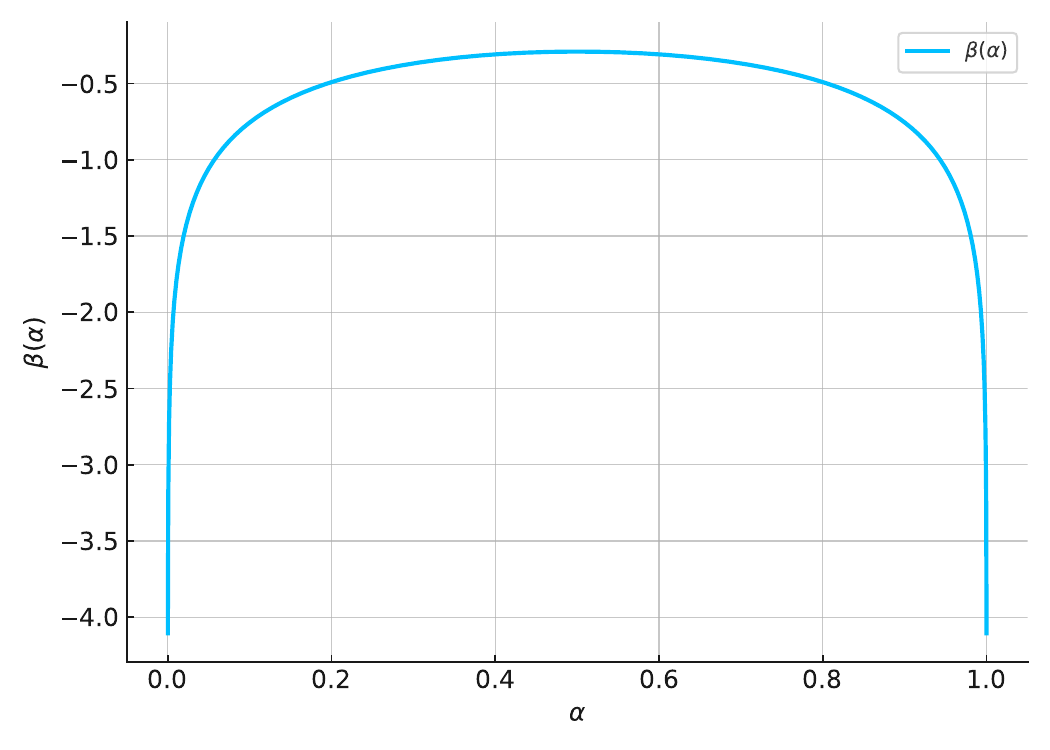}
    \caption{Plot of the function $\beta(\alpha)=\frac34 + \frac14 \frac{\alpha}{1-\alpha} \log(\alpha)+ \frac14 \frac{1-\alpha}{\alpha} \log(1-\alpha)+ \frac12 \log(\alpha(1-\alpha))$ over the interval from 0 to 1. The symmetry of the plot reflects the particle-hole duality of the injective norm of proposition \ref{prop:particle-hole-duality}.}
    \label{fig:plot}
\end{figure}
\end{proof}
\begin{lemma} \label{lem: Asympgamma}
    Let $0 \leq \alpha \leq 1$. Then, for every $d$, there exists a $\gamma_\alpha(d)$ such that:
    \begin{align}
        \frac{1}{2}\log(d)+\beta(\alpha)+ \Omega(\gamma_\alpha(d)) - \frac{\gamma^2_\alpha(d)}{2} = 0  , 
    \end{align}
with $\beta(\alpha)= \frac{1}{d^2 \alpha (1-\alpha)} \log K(\alpha d , d) - \frac12 \log(d)$, and for large $d$, $\beta(\alpha)= \frac34 + \frac14 \frac{\alpha}{1-\alpha} \log(\alpha)+ \frac14 \frac{1-\alpha}{\alpha} \log(1-\alpha) + \frac12 \log(\alpha(1-\alpha))+ \mathcal{O}\left(\frac{\log(d)}{d}\right)$. Furthermore, for large $d$, the solution scales as:
\begin{align}
    \gamma_\alpha(d)=\sqrt{\log d} +  \frac{\log \log d}{2 \sqrt{\log d}} + \frac{\beta(\alpha)}{\sqrt{\log d}} + o_{d \rightarrow \infty}(\frac
    1 {\sqrt{\log d}})  .
\end{align}
\end{lemma}
\proof 
The proof of this lemma follows that of lemma 3.20 in \cite{dartois2024injective}, with the substitution $p \rightarrow d$ and $d \rightarrow \alpha$.  \endproof
\begin{lemma} \label{lem: LimitDeterminant}
Let $\alpha$ be fixed, and let $\gamma(d)= o(\exp(d^{1-\delta}))$ be a sequence diverging to infinity for some $\delta > 0$:
\begin{align}
            \lim_{d \rightarrow \infty} \left[ \frac{1}{d^2 \alpha (1-\alpha)} \log \int_{\gamma(d)}^\infty e^{-\frac{N}{2}u^2} \E_\text{BHGAE}[\lvert \det(W-u) \rvert] \dd u  - \left( \Omega(\gamma(d))-\frac{\gamma^2(d)}{2} \right) \right] = 0 .
\end{align}
\end{lemma}
\proof
The key idea is to apply the Laplace method to the Kac-Rice integral. The details follow the proof of lemma 3.22 in \cite{dartois2024injective}. \endproof
\section{Complex Case}
The structure of the proof in the complex case closely parallels that of the real case. We propose an alternative but equivalent definition of the injective norm, which is more convenient to work with.  Moreover, we demonstrate that this definition admits an upper bound via the Kac-Rice formula.
 As in the real case, due to the homogeneity of the Hessian on the complex Grassmannian shown in lemma \ref{lem:homogeneity}, the Kac-Rice integral can be decomposed into two parts: a volume contribution and a determinant contribution arising from the integrand. Each of these contributions can then be evaluated analogously to the real case.\\
\begin{lemma} \label{eq-comp-def}
    Let T be a skew-symmetric d-dimensional random tensor of order $p$. The function $f_T(x^{(1)},\dots,x^{(p)})=\frac{|\langle T, x^{(1)}\wedge \ldots \wedge x^{(p)}\rangle|}{\lnorm x^{(1)}\wedge \ldots \wedge x^{(p)}\rnorm_2}$ has the same maximum as $g_T(x^{(1)},\dots,x^{(p)})=\frac{\Re{\langle T, x^{(1)}\wedge \ldots \wedge x^{(p)}\rangle}}{\lnorm x^{(1)}\wedge \ldots \wedge x^{(p)}\rnorm_2}$ and thus the injective norm of $T$ can be given as the maximum of $g_T$.
\end{lemma}
\begin{proof} First, observe that for any set of vectors $\{x^{(1)}, \dots , x^{(p)} \}$ in $\C^d$, we have $f_T(x^{(1)},  \dots , x^{(p)}) \geq g_T(x^{(1)}, \dots , x^{(p)}).$ Therefore, it suffices to show that there exists $\{y^{(i)}\}_i$ such that $g_T(y^{(1)}, \dots , y^{(p)})=\max f_T$. Let $\{x^{(i)}\}_i$ maximize  $f_T$ and let $\theta$ denote the phase of $\langle T, x^{(1)}\wedge \ldots \wedge x^{(p)}\rangle$ before taking the absolute value. Thus setting $y^{(1)}=e^{-i\theta}x^{(1)}$ and $y^{(j)}=x^{(j)}$ provides the set such that $g_T(y^{(1)}, \dots , y^{(p)})=\max f_T$. 
\end{proof}

\begin{lemma} \label{lem:asymptotics_constant_KR_complex}
    For fixed p and d and denoting $N=2p(d-p)$, and defining
    \begin{align}
        L(p,d)=\frac{\pi^{p(d-p)} \prod_{j=1}^p \Gamma(j)\prod_{j=1}^{d-p} \Gamma(j)}{\prod_{j=1}^d \Gamma(j)} \left({\frac{p(d-p)}{ 2\pi}}\right)^{p(d-p)+1} \label{eq:complex_volume}.
    \end{align}
Then we have:
\begin{align}
    \E\left[ \text{Crit}_{g_T}(D)\right] \leq L(p,d) \int_D e^{-Nu^2} \E_\text{cBHGAE}\left[|\det(V_N-u)|\right].
\end{align}
where $V_N$ is a $ \text{cBHGAE}(p,d,\frac{1}{2p(d-p)})$ matrix as defined in \ref{def:complex BHGAE}.
\end{lemma}

\proof This lemma closely resembles Proposition \ref{prop:upper-KR}. The only difference lies in the replacement of the real Grassmannian with the complex Grassmannian, which leads us to use $\text{Vol}(U(n))$ instead of $\text{Vol}(O(n))$, as shown in Appendix \ref{AppA}. Any other detail remains the same. 
\endproof

\subsection{Complex, fixed $p$, $d \rightarrow \infty$}
\begin{theorem}[Complex, skew-symmetric, $p$-fixed, $d \rightarrow \infty$] 
    Let $T\in \bigwedge^p\C^d$ be a random skew-symmetric tensor with i.i.d. $\mathcal{N}_\C(0,1)$ entries. Then  for every $\epsilon > 0 $, we have:
\begin{align}
    \limsup_{d \rightarrow \infty} \frac1{p(d-p)} \log \P \left( \frac{\injnorm{T}}{\sqrt{p(d-p)}} > E_0(p)+\epsilon \right) < 0.
\end{align}
Define $\ket{\psi_f}:=\frac{T}{\lnorm T\rnorm_{2}}$. Then, 
\begin{align}
\limsup_{d\rightarrow \infty}\frac1{p(d-p)}\log \P\left(\injnorm {\ket{\psi_f}}\ge\frac1{d^{\frac{p-1}{2}}}(\alpha(p)+\epsilon)\right)\le0.
\end{align}
\end{theorem}

\begin{lemma}[$d \rightarrow \infty $ limit of volume factor]
    In the limit $d \rightarrow \infty$ for a fixed $p$ the volume factor $L(p,d)$ tends toward:
    \begin{align}
        \lim_{d \rightarrow \infty} \frac1{2p(d-p)} \log L(p,d)= \frac{1+\log p}{2}.
    \end{align}
\begin{proof}
The proof of this lemma is very similar to that of Lemma \ref{lem:asymptotics_constant_KR} with the only difference in the fact that the non-$\Gamma$ contributions in eq. \eqref{eq:complex_volume} appear with a power of two compared to Lemma \ref{lem:asymptotics_constant_KR}, while the arguments of the $\Gamma$-contributions  are scaled by the factor of two. A straightforward calculation shows that these differences yield an overall factor of two in $\log L$ which is canceled by dividing by $2p(d-p)$ (as appropriate in the complex case) rather than $p(d-p)$. Thus, the limiting expression matches that of the real case.
\end{proof}
\end{lemma}
Next, we verify that all the necessary properties established in the real case also hold for a cBHGAE matrix.
\begin{proposition}\label{prop:complex_assumptions_check}
Let $H_N(u)=u-W_N$, where $W_N$ is a cBHGAE random matrix. Then, properties \ref{enum:prop_Wass}, \ref{enum:prop_Lip-Trace}, \ref{enum:prop_bound_determinant} are satisfied for all $u\in \R.$  Assuming $u\in \mathcal{K}\subset \mathcal{A}$, with $\mathcal{K}$ closed, property \ref{enum:prop_gap} is also satisfied.  
\end{proposition}

\begin{proof}
\textbf{Property 3 and 4} carry over directly from the real case. The Herbst argument for \textbf{property 2} requires  the boundedness of the operator norm of cBHGAE, which follows from proposition \ref{prop:complex_operator_bound}. The rest of the argument remains unchanged.  \\
Regarding \textbf{property 1}, in the complex case, the matrix Dyson equations are derived the same way as in proposition \ref{prop:assumptions_check}, using a method analogous to \cite{dartois2024injective}. For all $i<j$, let $\beta_{ij}^{(p)}:=E_{ij}^{(p)}-E_{ji}^{(p)}$ as for the real case, then the cBHGAE is defined as:
\begin{align*}
W_N={\frac1{\sqrt{4p(d-p)}}}\sum_{s=0,1}\sum_{1\le i<j\le p}\eta_s \otimes\beta_{ij}^{(p)}\otimes G^{(s)}_{ij}+(\eta_s \otimes\beta^{(p)}_{ij})^t\otimes G_{ij}^{(s)t},
\end{align*}
with $\eta_0=\begin{pmatrix}
1 & 0 \\
0 & -1
\end{pmatrix}$ and $\eta_1=\begin{pmatrix}
0 & 1 \\
1 & 0
\end{pmatrix}$. Furthermore, the 
$\mathscr{S}$ operator is given as follows:
\begin{align*}
\mathscr{S}_i[r]={\frac{1}{4p(d-p)}}\sum_{s=0,1}\sum_{k=1}^{d-p}\sum_{1\le i < j \le p}\eta_s \otimes \beta^{(p)}_{ij}r_k(\eta_s \otimes\beta^{(p)}_{ij)})^t+\eta_s\otimes\gamma^{(p)}_{ij}r_k(\eta_s \otimes\gamma^{(p)}_{ij})^t,
\end{align*}
with $\gamma_{ij}^{(p)}=-\beta_{ij}^{(p)}$ as in the real case. Using $\eta_s^2=\Id$, and choosing $r=(c\Id,\ldots, c\Id)$, we obtain:
\begin{align*}
    \mathscr{S}_i[c\Id]=\frac{p-1}{p}c\Id.
\end{align*}

Using equations \eqref{MDE1} and \eqref{MDE2} the Stieltjes transform can be given by:
\begin{equation*}
    m_p(z)=\frac{-z+\sqrt{z^2-4\left(\frac{p-1}{p}\right)}}{2\left( \frac{p-1}{p}\right)}.
\end{equation*}
The inverse of this Stieltjes transform yields the distribution $\rho_p(z)=\frac{p}{2\pi (p-1)}\sqrt{\left(4\left( \frac{p-1}{p}\right)-x^2\right)_+}.$ The remainder of the proof follows as in the proof of proposition \ref{prop:assumptions_check} and \cite[3.16]{dartois2024injective}.
\end{proof}
\subsection{Complex, double scaling}
The complex case with $p=\lfloor\alpha d\rfloor $ and $ d\rightarrow \infty$ is likewise analogous to the real case. We introduce a sequence whose asymptotic behavior determines the upper bound of interest. Then, applying Laplace's method, we show that the Kac-Rice integral converges to the same asymptotic limit.
To that end, we state the following lemmas  without proof, referring the reader to the corresponding arguments in the real case:
\begin{lemma}
    Fixing $\alpha$, for every $d$ there exists a unique $\gamma_\alpha(d) \equiv\gamma_\alpha^\C(d) > 0$ satisfying:
    \begin{align}
        \frac{\log d }{2} + \beta_\C(\alpha) + \Omega(\gamma_\alpha(d))- \frac{(\gamma_\alpha(d))^2}{2} = 0,
    \end{align}
where $\beta_\C(\alpha)$ is the same $\beta$ function in the real case. Thus, the asymptotic behavior of $\gamma(d)$ given by:
\begin{align}
    \gamma_\alpha(d)= \sqrt{\log d}+\frac{\log \log d}{2 \sqrt{\log d }} + \frac{\beta(\alpha)}{\log(d)} + o\left(\frac{1}{\log d}\right)
\end{align}
\end{lemma}
\noindent The next lemma describes the asymptotic behavior of the Kac-Rice integral: 
\begin{lemma}
    For a fixed $\alpha$ and $\gamma_\alpha(d)$ a divergent sequence such that for some $\delta > 0$:
    \begin{align}
        \gamma_\alpha(d)= o(\exp(d^{1-\delta}))
    \end{align}
Then, the following holds:
\begin{align}
        \lim_{d \rightarrow \infty} \left[ \frac{1}{d^2 \alpha (1-\alpha)} \log \int_{\gamma(d)}^\infty e^{-\frac{N}{2}u^2} \E_\text{BHGAE}[\lvert \det(W-u) \rvert] \dd u  - \left( \Omega(\gamma(d))-\frac{\gamma^2(d)}{2} \right) \right] = 0 
\end{align}
\end{lemma}
\noindent The next lemma presents the asymptotic behavior of the volume of $\text{Gr}_\C(\lfloor \alpha d \rfloor,d)$:

\begin{lemma}
    Let $p=\lfloor \alpha d\rfloor$ and $L(p , d)$ be defined as in \ref{lem:asymptotics_constant_KR_complex}. Then we have:
    \begin{align}
        \lim_{d \rightarrow \infty} \frac{1}{2p(d-p)} \log L(p ,d) = \frac{3}{4}+& \frac{1}{4} \frac{\alpha}{1-\alpha} \log(\alpha)+\frac{1}{4} \frac{1-\alpha}{\alpha} \log(1-\alpha) + \frac{1}{2} \log(d)\\ \nonumber & + \frac12 \log(\alpha(1-\alpha))+ \mathcal{O}\left( \frac{\log(d)}{d}\right).
    \end{align}
\end{lemma}
\proof
This argument closely follows that of Lemma \ref{lem: asym_real_Grass_vol}. Starting from the expression for the volume of the complex Grassmannian in Equation \eqref{eq: Complex_Grass_vol},  we note that the only difference from the real case is a factor of 2 arising from the arguments of the $\Gamma$ functions and the powers $\pi$. This factor of 2 is canceled by factor 2 in the denominator of $\frac{1}{2p(d-p)}$, giving the same limiting result as in the real case.
\endproof
\noindent And finally the upper bound in this case is given by:
\begin{theorem}[Complex, skew-symmetric, double scaling, $p,d \rightarrow \infty$] For a fixed $\alpha \in (0,1) $, let $p=\lfloor \alpha d\rfloor$ and $T\in \bigwedge^p\C^d$ be a random skew-symmetric tensor with entries that are i.i.d. $\mathcal{N}_\C(0,1)$ random variables, then we have for every $\epsilon > 0 $:
    \begin{align}
        \limsup_{d \rightarrow \infty} \frac{1}{d} \log \P\left(\frac{\injnorm{T}}{\sqrt{d^2 \alpha(1-\alpha)}} >\gamma_\alpha(d)+ \frac{\epsilon}{\sqrt{\log d}}\right) < 0
    \end{align}
Let $\ket{\psi_f}:=\frac{T}{\lnorm T\rnorm_{2}}$, then 
$$    \limsup_{d\rightarrow \infty}\frac1{p(d-p)}\log \P\left(\frac{\injnorm{ \ket{\psi_f}}}{\sqrt{d^2\alpha(1-\alpha)}}> d^{-\frac{\alpha d}{2}} e^{\frac12(\alpha + (1-\alpha)\log(1-\alpha))d} (\gamma_\alpha(d)+\epsilon)\right)<0$$
\end{theorem}
 \noindent The proof of this theorem follows that of the real case. It relies on complex analogues of Lemma \ref{lem: Asympgamma} and Lemma \ref{lem: LimitDeterminant}, which are stated above. Given the results established above, the proof proceeds identically to the real case and follows the argument in \cite{dartois2024injective}. 
\endproof
\section{Numerical Simulations}
To validate our analytical results in the case of finite $p$, we use the numerical methods implemented in \cite{fitter2022estimating}. 
These simulations are carried out for both real and complex skew-symmetric tensors, with tensor orders $p=2,3,4$. As anticipated, the results in both the real and complex settings closely match, supporting the asymptotic behavior that we obtain. Although simulations are performed for both real and complex tensors of orders $p=2,3$ and $4$, we display the complex-case results only for $p=3$, as a representative example. The behavior in the $p=2$ or $4$ cases is quantitatively indistinguishable from the real-case simulations, both in terms of convergence and the limiting bounds, and is therefore omitted to avoid redundancy. This reinforces the observation that the real and complex cases exhibit identical large-$d$ behavior in all tested orders. We do not extend the simulations to higher orders $p>4$, as the memory and runtime requirements grow prohibitively with both tensor order and ambient dimension, making large-scale sampling computationally infeasible in that regime.\\ In each case, we report the normalized injective norm of a random skew-symmetric tensor $T$, defined as
$$\frac{\injnorm{T}}{\sqrt{(d-p)}}.$$
This observable is expected to remain bounded in the large-$d$ limit. 
We employ two numerical approaches for estimating the injective norm, depending on the tensor order:
\begin{itemize}
    \item For  $p = 2$, the norm reduces to the largest singular value of the associated skew-symmetric matrix and can therefore be computed exactly via singular value decomposition.
    \item For $p > 2$, we use the gradient descent to obtain the approximate maximum value 
\end{itemize}
The number of samples used in the simulations varies with $p$  and $d$, primarily due to the memory demands of storing and optimizing high-order tensors:
\begin{itemize}
    \item For $p = 2$: 200 samples for all \( d \)
    \item For $p = 3$: 100 samples for $d \leq 250$, and 20 samples for $d > 250$ 
    \item For $p = 4$: 100 samples for $d \leq 70$, and 10 samples for $d > 70$
\end{itemize}
These choices reflect the computational cost of statistical precision, especially as both the rank and ambient dimension increase. However, in the large-dimensional limit, the self-averaging behavior of the injective norm leads to suppressed statistical fluctuations, partially mitigating the effect of reduced sample sizes.\\
We compare the simulation results against the theoretical upper bounds predicted by our analysis. These are given by:
\begin{align}
    \alpha(p)=\sqrt{p} E_0(p).
\end{align}
According to our analytical predictions, the values are ${\alpha(2)=2, \,\alpha(3)\approx2.870}$ and ${\alpha(4)\approx 3.588}$, indicated by blue lines in the plots.
\begin{figure}
    \centering
    \begin{subfigure}[t]{0.48\textwidth}
        \centering
        \includegraphics[width=\textwidth]{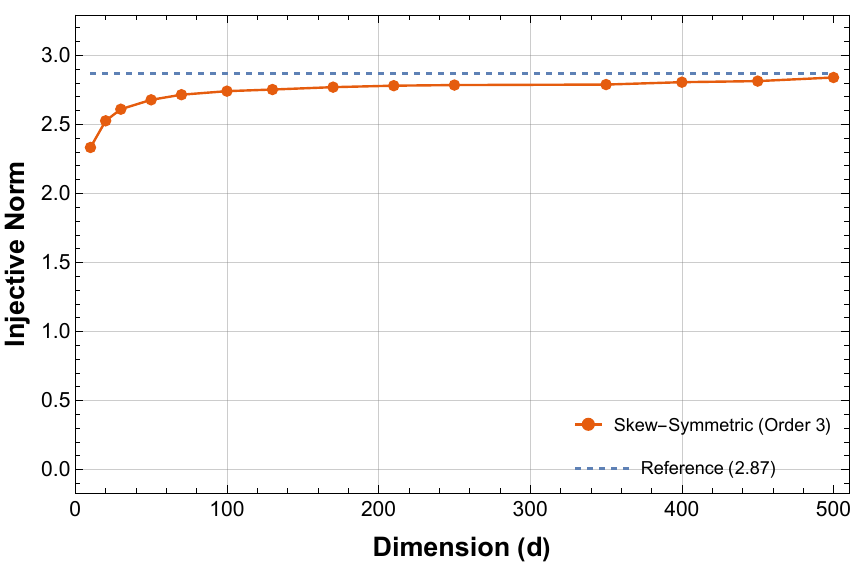}
        \caption{Order 3}
    \end{subfigure}%
    \hfill
    \begin{subfigure}[t]{0.48\textwidth}
        \centering
        \includegraphics[width=\textwidth]{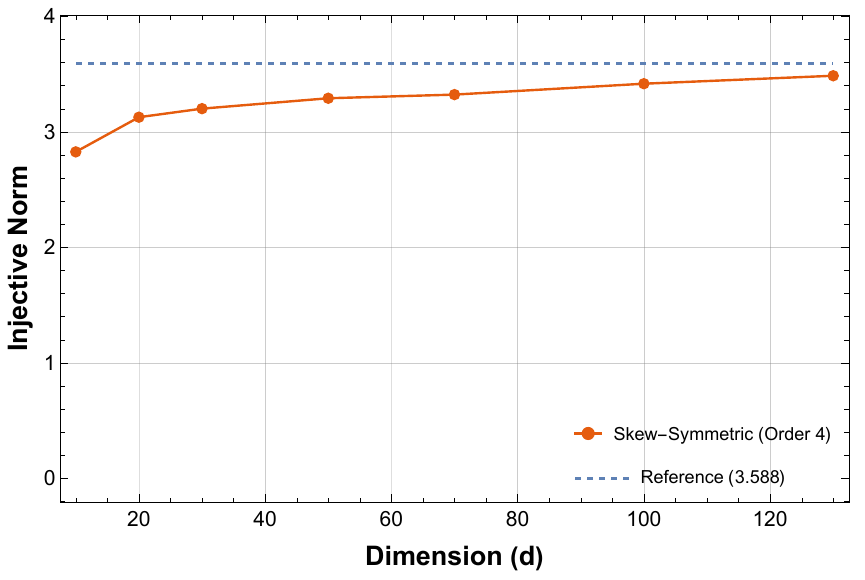}
        \caption{Order 4}
    \end{subfigure}\\[1ex]
    \begin{subfigure}[t]{0.48\textwidth}
        \centering
        \includegraphics[width=\textwidth]{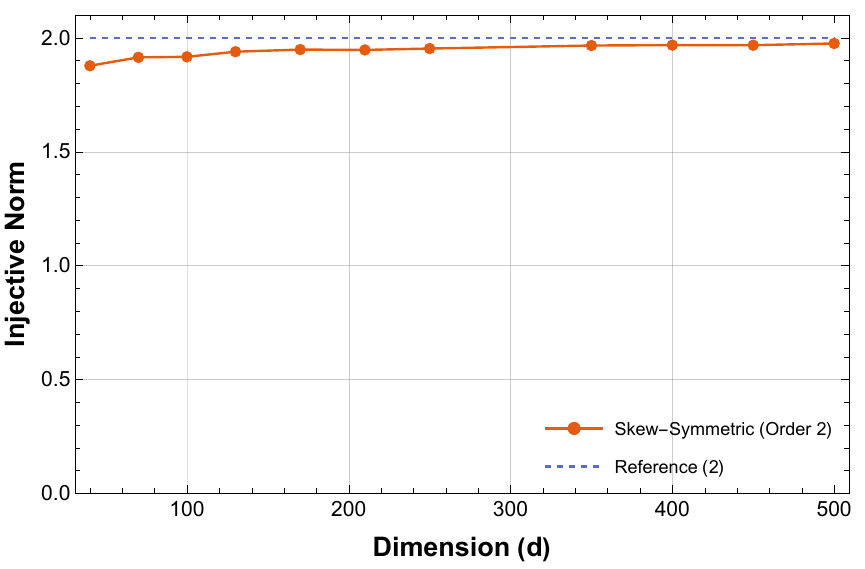}
        \caption{Order 2}
    \end{subfigure}%
    \hfill
    \begin{subfigure}[t]{0.48\textwidth}
        \centering
        \includegraphics[width=\textwidth]{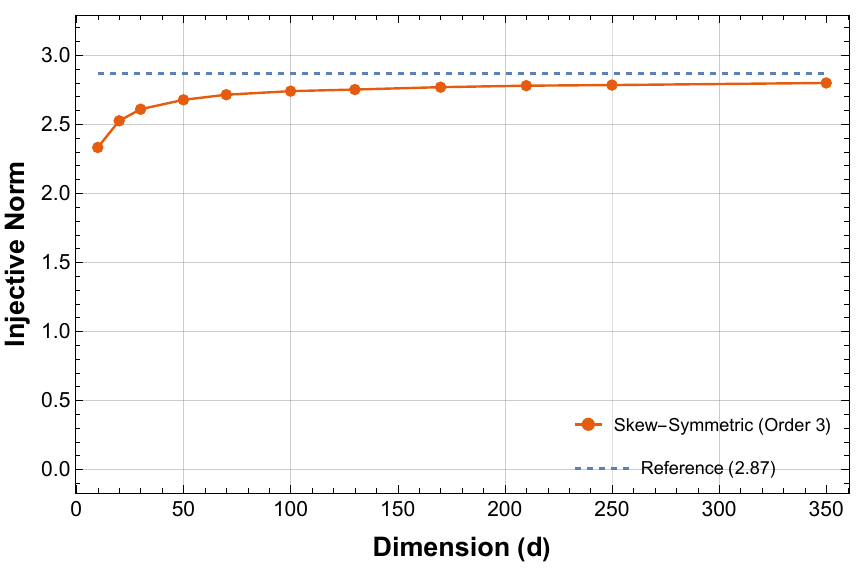}
        \caption{Order 3 (Complex)}
    \end{subfigure}
    \caption{Comparison of the values of the injective norm for random tensors of order 2, 3 and 4 as a function of the dimension d. The asymptotic analytical upper bound is indicated by the blue line.}
    \label{fig:order-comparison}
\end{figure}
The following figures show the normalized values and their corresponding ratios:\\
\begin{figure}
    \centering
    \begin{subfigure}[t]{0.48\textwidth}
        \centering
        \includegraphics[width=\textwidth]{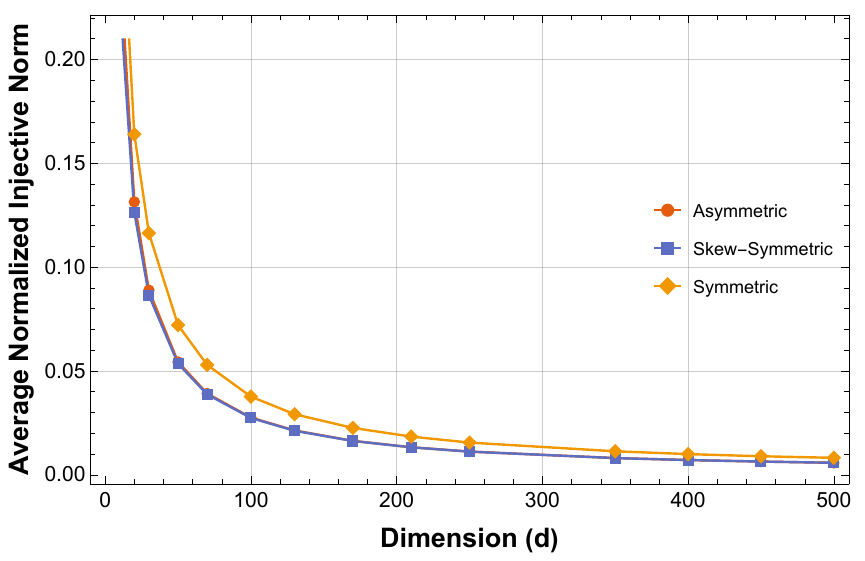}
        \caption{Injective norms of symmetric, skew- and, asymmetric order 3 normalized tensors}
    \end{subfigure}%
    \hfill
    \begin{subfigure}[t]{0.48\textwidth}
        \centering
        \includegraphics[width=\textwidth]{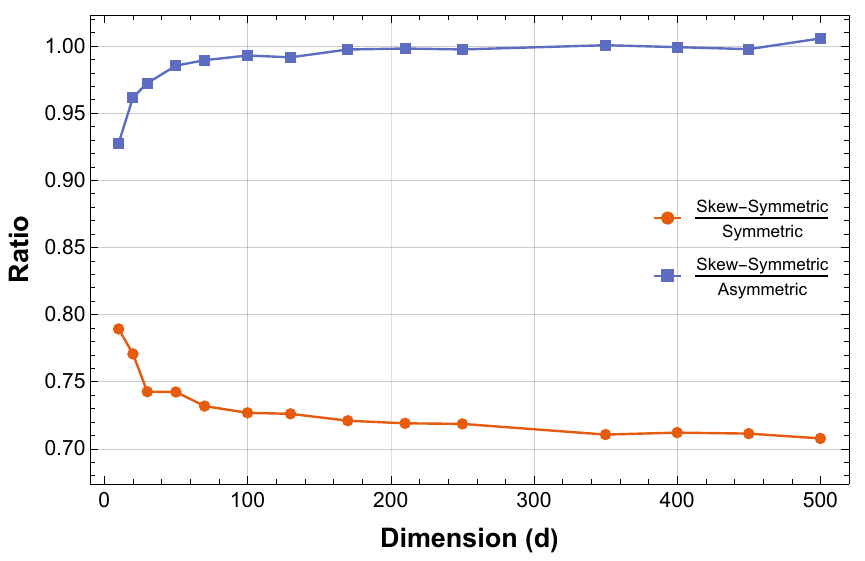}
        \caption{Ratios of normalized skew-symmetric injective norm against asymmetric and symmetric cases.}
    \end{subfigure}
    \caption{Comparison between the injective norm of normalized tensors and corresponding ratios.}
    \label{fig:normalized-ratios}
\end{figure}
\appendix
\section{Grassmannian Geometry} 
\label{AppA}
In this section, we briefly review the fundamental geometric properties of the Grassmann manifold relevant to our work. The topics discussed are mostly standard and well-known results and constructions. Our exposition closely follows the treatment provided in \cite{bendokat2024grassmann} and the references therein. 
We begin by presenting the normal coordinates on the Grassmannian around a reference point. We then discuss its structure as a homogeneous space and derive its volume, which will be used in the main text. \\
In general the Grassmannian $\text{Gr}(p,d)$ is defined as the set of all $p$-dimensional subspaces of the $d$-dimensional vector space $\mathbb{C}^d$ or $\mathbb{R}^d$.
\begin{definition}
Formally, for $\mathbb{K}=\R$ or $\C$, the Grassmannian can be defined as:
\begin{align}
    \text{Gr}(p,d)= \{ P \in \mathbb{K}^{d \times d}| P^T=P, P^2=P, \text{rank }P=p \}  .      
\end{align}
Equivalently, these projectors can be expressed for the real Grassmannian as:
\begin{align}
    \text{Gr}(p,d)=\Bigg\{O^T \begin{pmatrix}
    I_p & 0 \\
    0 & 0 \\  
\end{pmatrix}O \Bigg| O \in O(d) \Bigg\} \label{Grassmannian from North pole} ,
\end{align}
\end{definition}
\noindent with complex conjugate and $U(d)$ replacing transpose and $O(d)$ for the complex Grassmannian. The matrix $P_0=\begin{pmatrix}
    I_p & 0 \\
    0 & 0 \\ 
\end{pmatrix}$ plays a central role and will serve as our reference point on the Grassmannian in the subsequent discussion. $P_0$ represents the subspace spanned by the first $p$ basis vectors. This is reminiscent of the north pole $(1,0,\ldots,0)\in S^n$ being used in many works as a reference point.\cite{dartois2024injective}\\
\begin{lemma} \label{Volumeapp}
    The volume of the real Grassmann manifold is given as:
\begin{align}
    \text{Vol}(\text{Gr}_\R(p,d))=\frac{\text{Vol}(O(d))}{\text{Vol}(O(p))\text{Vol}(O(d-p))}.
\end{align}
The volume of the complex Grassmann manifold is given by:
\begin{align}
    \text{Vol}(\text{Gr}_\C(p,d))=\frac{\text{Vol}(U(d))}{\text{Vol}(U(p))\text{Vol}(U(d-p))}.
\end{align}
Explicitly, the volumes are:
\begin{align}
\label{eq: Real_Grass_vol}    \text{Vol}(\text{Gr}_\R(p,d))=\pi^{\frac{p(d-p)}{2}} \frac{\prod_{i=1}^{d-p}\Gamma(i/2)}{\prod_{i=p+1}^{d}\Gamma(i/2)}
\end{align}
\begin{align}
\label{eq: Complex_Grass_vol}
    \text{Vol}(\text{Gr}_\C(p,d))=\pi^{p(d-p)} \frac{\prod_{i=1}^{d-p}\Gamma(i)}{\prod_{i=p+1}^{d}\Gamma(i)}
\end{align}
\end{lemma}
\noindent A proof of this lemma can be found in \cite{nicolaescu2020lectures}, consistent with the construction of the Grassmannian as a homogeneous space.\\ 
\subsection{Normal coordinates and non-compact Stiefel manifold}
In particular, one can explicitly define normal coordinates around the $P_0$, i.e. the subspace spanned by the first $p$ canonical basis vectors. The tangent-space at this point is homeomorphic to $\mathbb{R}^{(d-p)\times p}$. Thus, any point in the tangent space can be represented as a $d-p$ by $p$ matrix $B \in \mathbb{R}^{(d-p)\times p}$. It is useful to define for each matrix $B \in \mathbb{R}^{(d-p)\times p} $ a matrix $\tilde{B} \in \mathbb{R}^{d\times p}$ with $\tilde{B}= \begin{pmatrix} I_p \\ B  \\ \end{pmatrix}$, \\
Furthermore, the coordinate map is given by 
\begin{align}
    \varphi: \mathcal{U} \rightarrow \mathbb{R}^{d \times d}, \quad B \mapsto \tilde{B} (\tilde{B}^T\tilde{B})^{-1} \tilde{B}^T
\end{align}
where $\mathcal{U}$ is an open neighborhood around zero in $\mathbb{R}^{(d-p) \times p}$.
Here, each point on the Grassmannian is represented by the orthogonal projector onto the corresponding subspace.
Note that this projection matrix corresponds precisely to the subspace spanned by the columns of $\tilde{B}$.\\
Since the integrand in the Kac-Rice formula is independent of the choice of point (as shown in Lemma \ref{lem:homogeneity}) we can focus instead on the pullback of $f$ from \eqref{eq:definition_f} by the map $\varphi$, $\varphi^* f(B)=f(\varphi(B))$ in $\mathbb{R}^{(d-p) \times p}$ and evaluate it as such in the neighborhood $\mathcal{U}$, and multiply the result by the volume of the Grassmannian. In particular, $\varphi^* f(B)=f(\varphi(B))$ can be written as:
\begin{align}
    \varphi^* f_T(B) = \frac{1}{\sqrt{p(d-p)}}\frac{T(\tilde{b}_1,\dots,\tilde{b}_k)}{\sqrt{\det(\tilde{B}^T\tilde{B})}}
\end{align}
with $\tilde{b}_i$ being the $i$-th column of $\tilde{B}$.
\section{Computation of Hessian and other correlations}\label{AppB}

We begin with the \textbf{real case} and, for convenience, write $f$ instead of $\phi^*f$ under an abuse of notation. On the tangent space of the point $P_0$, let $X,Y \in \mathbb{R}^{(d-p) \times p} $ we  have $\tilde{X}= \begin{pmatrix} I_p \\ X  \\ \end{pmatrix}$ and $\tilde{Y}= \begin{pmatrix} I_p \\ Y  \\ \end{pmatrix}$ whose columns span the corresponding vector subspace of $\phi(X)$ and $\phi(Y   )$. The expectation value $\mathbb{E}\left[ f(X)f(Y) \right]$ is then given as:
\begin{align}
    \mathbb{E}\left[ f(X)f(Y)    \right] =  \frac{1}{p(d-p)}\frac{\det[\tilde{X}^t\tilde{Y}]}{\sqrt{\det[\tilde{X}^t\tilde{X}]\det[\tilde{Y}^t\tilde{Y}]}} \label{covariance_of_f}
\end{align} 

\begin{lemma} \label{lem:homogeneity}
Considering the integrand of the Kac-Rice formula derived from the function $f_T$ on the Grassmann manifold $\text{Gr}(p,d)$:
\begin{align}
    J(x) \equiv  \E(\lvert \det \text{Hess }f(x)\rvert \mathbf{1}_{f(x)\in [t_1,t_2]}:\nabla f(x)=0)\rho_{\nabla f(x)}(0),
\end{align}
then function $J: \text{Gr}(p,d) \rightarrow \R$ is  constant  across the Grassmann manifold. 
\end{lemma}
\proof
First we define the Hessian as a $(1,1)$-tensor on the manifold. $(\text{Hess} f)_{ij} = \nabla_i (\text{grad} \, f_T)_j $  where $\nabla$ is the Levi-Civita connection on the manifold. The determinant of the Hessian, defined in this way, is invariant under any coordinate transformation. Furthermore, the condition $\nabla f\equiv \text{grad} \, f=0$ is invariant under any change of coordinate. Finally it remains to check that the covariance of function $f_T$ is invariant under the action of $O(d)$ which transfers one point of the Grassmannian to another. This is clear from the equation \eqref{covariance_of_f} and its invariance under $\tilde{X} \mapsto O\tilde{X}$ and $\tilde{Y} \mapsto O\tilde{Y}$. 
\endproof
\noindent{\bf Correlations of the random function $f$ and its derivatives:} Taking the derivative with respect to one of the coordinates yields the following tangent-function correlation:
\begin{align}
    \mathbb{E}\left[\frac{\partial f(X)}{\partial x^{(a)}_i} f(Y)\right] \Bigg|_{X=Y=0} = 0
\end{align}
The next relevant correlation is:
\begin{align}
    \mathbb{E}\left[\frac{\partial f(X)}{\partial x^{(a)}_i} \frac{ \partial f(Y)}{\partial y^{(b)}_j}\right]\Bigg|_{X=Y=0} =  \frac{1}{p(d-p)}\delta^{ab} \delta_{ij}
\end{align}
We now consider the correlation between the Hessian and the gradient:
\begin{align}
    \mathbb{E}\left[\frac{\partial^2 f(X)}{\partial x^{(a)}_i\partial x^{(b)}_j} \frac{\partial f(Y)}{\partial y^{(c)}_k}\right]\Bigg|_{X=Y=0}=0
\end{align}
Finally, we compute the covariance of the Hessian entries:
\begin{multline}\label{eq:hessian_cov}
    \mathbb{E}\left[\frac{\partial^2 f(X)}{\partial x^{(a)}_i\partial x^{(b)}_j} \frac{\partial f(Y)}{\partial y^{(c)}_k\partial y^{(d)}_l}\right]\Bigg|_{X=Y=0}=\frac{1}{p(d-p)}\times\\
    \left(\delta^{ac}\delta^{bd} \delta_{ik}\delta_{jl}+\delta^{ad}\delta^{bc} \delta_{il}\delta_{jk} -\delta^{ac}\delta^{bd} \delta_{il}\delta_{jk}-\delta^{ad}\delta^{bc} \delta_{ik}\delta_{jl}+\delta^{ab}\delta^{cd} \delta_{ij}\delta_{kl}\right)
\end{multline}
Furthermore for \textbf{the complex case}, one can show that:\\
$\E(\Re{f(X)}\Re{f(Y)})=\frac12\left(\Re{\E(f(X)\overline{f(Y)})}+\Re{\E(f(X)f(Y))} \right)$, where due to the choice of the distributions the second term vanishes.
Now we compute the correlation functions of  derivatives of $\Re{f(X)}$, following the same procedure as in the real case, up to second order. Note, however, that in the complex case, derivatives can be taken with respect to the real or imaginary parts of each variable. To clarify this, we define:
\begin{align*}
    x_i^a=\alpha_i^a+i \beta_i^a \\
    y_i^a=\sigma_i^a+i \tau_i^a
\end{align*}
Accordingly, the correlation function becomes:
\begin{align}
    \E[\Re{f(X)}\Re{f(Y)}]=\frac{1}{2} \frac{\Re{\det[\tilde{X}^\dagger Y]}}{\sqrt{\det[\tilde{X}^\dagger X]\det[\tilde{Y}^\dagger Y]}}
\end{align}
Here, we have used the same coordinate charts as in the real case to define  $X$ and $Y$.\\
As in the real case, the gradient and Hessian of the function can be expressed as centered Gaussian random variables, with the correlation structure given below. For convenience, we write $g=\Re f$ :
\begin{align}
    &\E\left[g(X)\frac{\partial}{\partial \tau_i^a} g(Y)\right]=\E\left[g(X)\frac{\partial}{\partial \sigma_i^a} g(Y)\right]=0 \\
    &\E\left[\frac{\partial}{\partial \alpha_j^b}g(X)\frac{\partial}{\partial \tau_i^a} g(Y)\right]=0 \\ &
    \E\left[\frac{\partial}{\partial \beta_j^b}g(X)\frac{\partial}{\partial \tau_i^a} g(Y)\right]=\E\left[\frac{\partial}{\partial \alpha_j^b}g(X)\frac{\partial}{\partial \sigma_i^a} g(Y)\right]= \frac{1}{2p(d-p)}\delta_{ij} \delta^{ab}\\ & \E\left[g(X)\frac{\partial}{\partial \tau_j^b} \frac{\partial}{\partial \tau_i^a} g(Y)\right]=\E\left[g(X)\frac{\partial}{\partial \sigma_j^b} \frac{\partial}{\partial \sigma_i^a} g(Y)\right]=-\frac{1}{2p(d-p)}\delta_{ij} \delta^{ab} \\
    & \E\left[\frac{\partial}{\partial \alpha_k^c}g(X)\frac{\partial}{\partial \tau_j^b} \frac{\partial}{\partial \tau_i^a} g(Y)\right]=\E\left[\frac{\partial}{\partial \beta_k^c}g(X)\frac{\partial}{\partial \sigma_j^b} \frac{\partial}{\partial \sigma_i^a} g(Y)\right]=0 \\
    & \E\left[\frac{\partial}{\partial \alpha_l^d}\frac{\partial}{\partial \alpha_k^c}g(X)\frac{\partial}{\partial \tau_j^b} \frac{\partial}{\partial \tau_i^a} g(Y)\right]=\E\left[\frac{\partial}{\partial \beta_l^d}\frac{\partial}{\partial \beta_k^c}g(X)\frac{\partial}{\partial \sigma_j^b} \frac{\partial}{\partial \sigma_i^a} f(Y)\right]  \\ \nonumber &  = \frac{1}{2p(d-p)}\left(\delta^{ac}\delta^{bd} \delta_{ik}\delta_{jl}+\delta^{ad}\delta^{bc} \delta_{il}\delta_{jk} -\delta^{ac}\delta^{bd} \delta_{il}\delta_{jk}-\delta^{ad}\delta^{bc} \delta_{ik}\delta_{jl}\right)\\
    & \E\left[ \frac{\partial}{\partial \beta_l^d}\frac{\partial}{\partial \beta_k^c}g(X)\frac{\partial}{\partial \tau_j^b} \frac{\partial}{\partial \tau_i^a} g(Y) \right] = \E\left[\frac{\partial}{\partial \alpha_l^d}\frac{\partial}{\partial \alpha_k^c}g(X)\frac{\partial}{\partial \sigma_j^b} \frac{\partial}{\partial \sigma_i^a} f(Y)\right]  \\ \nonumber &  =  \frac{1}{2p(d-p)}\left(\delta^{ac}\delta^{bd} \delta_{ik}\delta_{jl}+\delta^{ad}\delta^{bc} \delta_{il}\delta_{jk} -\delta^{ac}\delta^{bd} \delta_{il}\delta_{jk}-\delta^{ad}\delta^{bc} \delta_{ik}\delta_{jl}+\delta^{ab}\delta^{cd} \delta_{ij}\delta_{kl}\right).
\end{align}

\newpage


\bibliographystyle{alpha}
\bibliography{references}

\newcommand{\etalchar}[1]{$^{#1}$}
\begin{thebibliography}{OWBVdN14}

\bibitem[ABA{\v{C}}13]{auffinger2013random}
Antonio Auffinger, G{\'e}rard Ben~Arous, and Ji{\v{r}}{\'\i} {\v{C}}ern{\`y}.
\newblock Random matrices and complexity of spin glasses.
\newblock {\em Communications on Pure and Applied Mathematics}, 66(2):165--201, 2013.

\bibitem[AEKN19]{alt2019location}
Johannes Alt, L{\'a}szl{\'o} Erd{\"o}s, Torben~H Kr{\"u}ger, and Yuriy Nemish.
\newblock Location of the spectrum of kronecker random matrices.
\newblock In {\em Annales de l'institut Henri Poincare}, volume~55, 2019.

\bibitem[AGZ10]{anderson2010introduction}
Greg~W. Anderson, Alice Guionnet, and Ofer Zeitouni.
\newblock {\em An introduction to random matrices}, volume 118 of {\em Cambridge Studies in Advanced Mathematics}.
\newblock Cambridge University Press, Cambridge, 2010.

\bibitem[AMM10a]{aulbach2010geometric}
Martin Aulbach, Damian Markham, and Mio Murao.
\newblock Geometric entanglement of symmetric states and the majorana representation.
\newblock In {\em Conference on Quantum Computation, Communication, and Cryptography}, pages 141--158. Springer, 2010.

\bibitem[AMM10b]{aulbach2010maximally}
Martin Aulbach, Damian Markham, and Mio Murao.
\newblock The maximally entangled symmetric state in terms of the geometric measure.
\newblock {\em New Journal of Physics}, 12(7):073025, 2010.

\bibitem[AS17]{aubrun2017alice}
Guillaume Aubrun and Stanis{\l}aw~J Szarek.
\newblock {\em Alice and Bob meet Banach}, volume 223.
\newblock American Mathematical Soc., 2017.

\bibitem[BABM22]{arous2022exponential}
G\'{e}rard Ben~Arous, Paul Bourgade, and Benjamin McKenna.
\newblock Exponential growth of random determinants beyond invariance.
\newblock {\em Probab. Math. Phys.}, 3(4):731--789, 2022.

\bibitem[BABM24]{arous2024landscape}
G\'{e}rard Ben~Arous, Paul Bourgade, and Benjamin McKenna.
\newblock Landscape complexity beyond invariance and the elastic manifold.
\newblock {\em Comm. Pure Appl. Math.}, 77(2):1302--1352, 2024.

\bibitem[BBvH23]{bandeira2023matrix}
Afonso~S Bandeira, March~T Boedihardjo, and Ramon van Handel.
\newblock Matrix concentration inequalities and free probability.
\newblock {\em Inventiones mathematicae}, 234(1):419--487, 2023.

\bibitem[BGJ{\etalchar{+}}24]{bandeira2024geometric}
Afonso~S Bandeira, Sivakanth Gopi, Haotian Jiang, Kevin Lucca, and Thomas Rothvoss.
\newblock A geometric perspective on the injective norm of sums of random tensors.
\newblock {\em arXiv preprint arXiv:2411.10633}, 2024.

\bibitem[BHK21]{BianchiKieburg2021}
Eugenio Bianchi, Lucas Hackl, and Mario Kieburg.
\newblock Page curve for fermionic gaussian states.
\newblock {\em Phys. Rev. B}, 103:L241118, Jun 2021.

\bibitem[Boe24]{boedihardjo2024injective}
March~T Boedihardjo.
\newblock Injective norm of random tensors with independent entries.
\newblock {\em arXiv preprint arXiv:2412.21193}, 2024.

\bibitem[BS25]{bates2025balanced}
Erik Bates and Youngtak Sohn.
\newblock Balanced multi-species spin glasses.
\newblock {\em arXiv preprint arXiv:2507.06522}, 2025.

\bibitem[BZA24]{bendokat2024grassmann}
Thomas Bendokat, Ralf Zimmermann, and P-A Absil.
\newblock A grassmann manifold handbook: Basic geometry and computational aspects.
\newblock {\em Advances in Computational Mathematics}, 50(1):1--51, 2024.

\bibitem[DLSST25]{delporte2025real}
Nicolas Delporte, Giacomo La~Scala, Naoki Sasakura, and Reiko Toriumi.
\newblock Real eigenvalue/vector distributions of random real antisymmetric tensors.
\newblock {\em arXiv preprint arXiv:2510.20398}, 2025.

\bibitem[DM24]{dartois2024injective}
Stephane Dartois and Benjamin McKenna.
\newblock Injective norm of real and complex random tensors i: From spin glasses to geometric entanglement.
\newblock {\em arXiv preprint arXiv:2404.03627}, 2024.

\bibitem[DNT22]{dartois2022entanglement}
Stephane Dartois, Ion Nechita, and Adrian Tanasa.
\newblock Entanglement criteria for the bosonic and fermionic induced ensembles.
\newblock {\em Quantum Information Processing}, 21(11):376, 2022.

\bibitem[DZ25]{dartois2025injectivenormcssquantum}
Stephane Dartois and Gilles Zémor.
\newblock The injective norm of css quantum error-correcting codes.
\newblock {\em arXiv preprint arXiv:2510.23736}, 2025.

\bibitem[FLN22]{fitter2022estimating}
Khurshed Fitter, Cecilia Lancien, and Ion Nechita.
\newblock Estimating the entanglement of random multipartite quantum states.
\newblock {\em arXiv preprint arXiv:2209.11754}, 2022.

\bibitem[For10]{forrester2010log}
Peter~J Forrester.
\newblock {\em Log-gases and random matrices (LMS-34)}.
\newblock Princeton university press, 2010.

\bibitem[GKM12]{grabowski2012segre}
Janusz Grabowski, Marek Ku{\'s}, and Giuseppe Marmo.
\newblock Segre maps and entanglement for multipartite systems of indistinguishable particles.
\newblock {\em Journal of Physics A: Mathematical and Theoretical}, 45(10):105301, 2012.

\bibitem[HW23]{huang2023entropy}
Youyi Huang and Lu~Wei.
\newblock Entropy fluctuation formulas of fermionic gaussian states.
\newblock In {\em Annales Henri Poincar{\'e}}, volume~24, pages 4283--4342. Springer, 2023.

\bibitem[Kur10]{kuriki2010distributions}
Satoshi Kuriki.
\newblock Distributions of the largest singular values of skew-symmetric random matrices and their applications to paired comparisons.
\newblock {\em Communications in Statistics—Theory and Methods}, 39(8-9):1522--1535, 2010.

\bibitem[LM00]{laurent2000adaptive}
Beatrice Laurent and Pascal Massart.
\newblock Adaptive estimation of a quadratic functional by model selection.
\newblock {\em Annals of Statistics}, pages 1302--1338, 2000.

\bibitem[McK24]{mckenna2024complexity}
Benjamin McKenna.
\newblock Complexity of bipartite spherical spin glasses.
\newblock {\em Ann. Inst. Henri Poincar\'{e} Probab. Stat.}, 60(1):636--657, 2024.

\bibitem[Meh04]{mehta2004random}
Madan~Lal Mehta.
\newblock {\em Random matrices}, volume 142.
\newblock Elsevier, 2004.

\bibitem[NDT10]{nguyen2010tensor}
Nam~H Nguyen, Petros Drineas, and Trac~D Tran.
\newblock Tensor sparsification via a bound on the spectral norm of random tensors.
\newblock {\em arXiv preprint arXiv:1005.4732}, 2010.

\bibitem[Nic20]{nicolaescu2020lectures}
Liviu~I Nicolaescu.
\newblock {\em Lectures on the Geometry of Manifolds}.
\newblock World Scientific, 2020.

\bibitem[OWBVdN14]{orus2014geometric}
Rom{\'a}n Or{\'u}s, Tzu-Chieh Wei, Oliver Buerschaper, and Maarten Van~den Nest.
\newblock Geometric entanglement in topologically ordered states.
\newblock {\em New Journal of Physics}, 16(1):013015, 2014.

\bibitem[Pis03]{Pisier_2003}
Gilles Pisier.
\newblock {\em Introduction to Operator Space Theory}.
\newblock London Mathematical Society Lecture Note Series. Cambridge University Press, 2003.

\bibitem[PS24]{pastur2024entanglement}
Leonid Pastur and Victor Slavin.
\newblock Entanglement entropy of free fermions with a random matrix as a one-body hamiltonian.
\newblock {\em Entropy}, 26(7):564, 2024.

\bibitem[Sas24]{sasakura2024signed}
Naoki Sasakura.
\newblock Signed eigenvalue/vector distribution of complex order-three random tensor.
\newblock {\em Progress of Theoretical and Experimental Physics}, 2024(5):053A04, 2024.

\bibitem[SG24]{steinberg2024finding}
Jonathan Steinberg and Otfried G{\"u}hne.
\newblock Finding maximal quantum resources.
\newblock {\em Physical Review A}, 110(6):062428, 2024.

\bibitem[Shi95]{shimony1995degree}
Abner Shimony.
\newblock Degree of entanglement a.
\newblock {\em Annals of the New York Academy of Sciences}, 755(1):675--679, 1995.

\bibitem[Sto25]{stojnic2025ground}
Mihailo Stojnic.
\newblock Ground state energies of multipartite $ p $-spin models--partially lifted rdt view.
\newblock {\em arXiv preprint arXiv:2509.05916}, 2025.

\bibitem[Sub23]{Sub2023}
Eliran Subag.
\newblock {TAP approach for multispecies spherical spin glasses II: The free energy of the pure models}.
\newblock {\em The Annals of Probability}, 51(3):1004 -- 1024, 2023.

\bibitem[TS14]{tomioka2014spectral}
Ryota Tomioka and Taiji Suzuki.
\newblock Spectral norm of random tensors.
\newblock {\em arXiv preprint arXiv:1407.1870}, 2014.

\bibitem[WG03]{wei2003geometric}
Tzu-Chieh Wei and Paul~M Goldbart.
\newblock Geometric measure of entanglement and applications to bipartite and multipartite quantum states.
\newblock {\em Physical Review A}, 68(4):042307, 2003.

\end{thebibliography}

\end{document}